\documentclass{compositio}

\usepackage{amsfonts}
\usepackage{mathrsfs}
\usepackage{amsmath}
\usepackage{amsmath,graphics}
\usepackage{amssymb}
\usepackage{indentfirst,latexsym,bm,amsmath,pstricks,amssymb,amsthm,graphicx,fancyhdr,float,color}
\usepackage[all]{xy}
\usepackage{amsthm}
\usepackage{amscd}
\usepackage{hyperref}
\usepackage[all]{hypcap}

\theoremstyle{definition}
  \newtheorem{definition}[subsubsection]{Definition}
  
  \newtheorem{definition-proposition}[subsubsection]{Definition-Proposition}
  \newtheorem{definition-lemma}[subsubsection]{Definition-Lemma}
  \newtheorem{example}[subsubsection]{Example}
  \newtheorem{remark}[subsubsection]{Remark}
  \newtheorem{remarks}[subsubsection]{Remarks}

\theoremstyle{plain}
  \newtheorem{theorem}[subsubsection]{Theorem}
  \newtheorem{proposition}[subsubsection]{Proposition}
  \newtheorem{lemma}[subsubsection]{Lemma}
  \newtheorem{corollary}[subsubsection]{Corollary}
  \newtheorem{conjecture}[subsubsection]{Conjecture}

\numberwithin{equation}{subsection}
\renewcommand{\theequation}{\arabic{subsection}-\arabic{equation}}

\def\wt{\widetilde}
\def\ol{\overline}
\def\lra{\longrightarrow}

\def\Aut{\text{\rm{Aut\,}}}
\def\({$($}
\def\){$)$}

\def\bbp{\mathbb P}
\def\cala{\mathcal A}

\def\cale{\mathcal E}

\def\call{\mathcal L}
\def\calm{\mathcal M}

\def\calo{\mathcal O}
\def\calq{\mathcal Q}

\def\calt{\mathcal T}

\def\rank{\text{{\rm rank\,}}}

\def\osb{\omega_{\ol S/\ol B}}
\def\fosb{\bar f_*\omega_{\ol S/\ol B}}
\def\olb{\ol B}
\def\olc{\ol C}
\def\ols{\ol S}

\newcommand{\adele}{{{\mathbb{A}}_f}}
\newcommand{\Jac}{\mathrm{Jac}}

\newcommand{\Qbb}{\mathbb{Q}}
\newcommand{\Rbb}{\mathbb{R}}
\newcommand{\Zbb}{\mathbb{Z}}
\newcommand{\Zbhat}{\hat{\Zbb}}

\newcommand{\Mcal}{\mathcal{M}}
\newcommand{\Acal}{\mathcal{A}}
\newcommand{\Hscr}{\mathscr{H}}
\newcommand{\Cbb}{\mathbb{C}}
\newcommand{\la}{\leftarrow}
\newcommand{\ra}{\rightarrow}
\newcommand{\mono}{\hookrightarrow}
\newcommand{\Tcal}{\mathcal{T}}
\newcommand{\GSp}{\mathrm{GSp}}
\newcommand{\Sp}{\mathrm{Sp}}
\newcommand{\bsh}{\backslash}
\newcommand{\isom}{\simeq}
\newcommand{\Urm}{\mathrm{U}}
\newcommand{\Gbb}{\mathbb{G}}
\newcommand{\mrm}{\mathrm{m}}
\newcommand{\Gbf}{\mathbf{G}}
\newcommand{\Hbf}{\mathbf{H}}
\newcommand{\Xcal}{\mathcal{X}}
\newcommand{\Sbb}{\mathbb{S}}
\newcommand{\Res}{\mathrm{Res}}
\newcommand{\ad}{\mathrm{ad}}
\newcommand{\Ad}{\mathrm{Ad}}
\newcommand{\GL}{\mathrm{GL}}
\newcommand{\gfrak}{\mathfrak{g}}
\newcommand{\Lie}{\mathrm{Lie}}
\newcommand{\SL}{\mathrm{SL}}
\newcommand{\ibf}{\mathbf{i}}
\newcommand{\ot}{\overset}
\newcommand{\der}{\mathrm{der}}
\newcommand{\inv}{{-1}}

\newcommand{\Ker}{\mathrm{Ker}}
\newcommand{\Vbb}{\mathbb{V}}
\newcommand{\Acalbar}{\overline{\Acal}}

\newcommand{\Ocal}{\mathcal{O}}
\newcommand{\End}{\mathrm{End}}
\newcommand{\SU}{\mathbf{SU}}

\newcommand{\xbar}{\bar{x}}
\newcommand{\Spec}{\mathrm{Spec}\,}
\newcommand{\tr}{\mathrm{tr}}
\newcommand{\ifof}{if and only if\ }
\newcommand{\wrt}{with respect to\ }
\newcommand{\cosg}{compact open subgroup\ }

\newcommand{\Ubf}{\mathbf{U}}
\newcommand{\Pbb}{\mathbb{P}}
\newcommand{\abar}{\bar{a}}

\newcommand{\Ecal}{\mathcal{E}}
\newcommand{\Ucal}{{\mathcal{U}}}

\newcommand{\Cbar}{{\overline{C}}}
\newcommand{\Mbar}{{\overline{M}}}
\newcommand{\Vcal}{{\mathcal{V}}}
\newcommand{\Tbf}{{\mathbf{T}}}
\newcommand{\Xfrak}{\mathfrak{X}}
\newcommand{\Hom}{\mathrm{Hom}}
\newcommand{\Qac}{\overline{\mathbb{Q}}}
\newcommand{\GSpin}{\mathbf{GSpin}}
\newcommand{\Spin}{\mathbf{Spin}}
\newcommand{\Gr}{\mathrm{Gr}}
\newcommand{\Nbb}{\mathbb{N}}
\newcommand{\SO}{\mathbf{SO}}

\newcommand{\GU}{{\mathbf{GU}}}
\newcommand{\id}{\mathrm{id}}

\newcommand{\kbar}{{\bar{k}}}
\newcommand{\Fcal}{{\mathcal{F}}}

\newcommand{\fbar}{{\overline{f}}}
\newcommand{\Bbar}{{\overline{B}}}
\newcommand{\Sbar}{{\overline{S}}}
\newcommand{\minuszero}{{-1,0}}
\newcommand{\zerominus}{{0,-1}}
\newcommand{\Rep}{\mathrm{Rep}}

\newcommand{\Std}{{\mathrm{Std}}}
\newcommand{\Hcal}{\mathcal{H}}

\newcommand{\Nbf}{\mathbf{N}}
\newcommand{\Zbf}{\mathbf{Z}}
\newcommand{\MT}{\mathbf{MT}}

\begin{document}

\title[Oort conjecture for  Shimura varieties of unitary and orthogonal types]{On the Oort conjecture for  Shimura varieties of unitary and orthogonal types}
\author{Ke Chen}
\email{kechen@ustc.edu.cn}
\address{Department of mathematics, Nanjing University, Nanjing, Jiangsu Province, China, 210093; \space School of mathematics, University of Science and Technology of
China, Hefei, China, 230026}

\author{Xin Lu}
\email{luxin001@uni-mainz.de}
\address{Department of Mathematics, East China Normal University, Shanghai, China, 200241}
\curraddr{Institut f\"ur Mathematik, Universit\"at Mainz, Mainz, Germany, 55099}


\author{Kang Zuo}
\email{zuok@uni-mainz.de}
\address{Institut f\"ur Mathematik, Universit\"at Mainz, Mainz, Germany, 55099}

\dedication{In Memory of Professor Gang Xiao}

\begin{abstract}
In this paper we study the Oort conjecture concerning the non-existence of Shimura subvarieties contained generically in the Torelli locus in the Siegel modular variety $\Acal_g$. Using the poly-stability of Higgs bundles on curves and the slope inequality of Xiao on fibred surfaces, we show that a Shimura curve $C$ is not contained generically in the Torelli locus if its canonical Higgs bundles contains a unitary Higgs subbundle of rank at least $(4g+2)/5$. From this we prove that a Shimura subvariety of $\SU(n,1)$-type is not contained generically in the Torelli locus when a numerical inequality holds,
which involves the genus $g$, the dimension $n+1$, the degree $2d$ of CM field of the Hermitian space, and the type of the  symplectic representation defining the Shimura subdatum. A similar result holds for Shimura subvarieties of $\SO(n,2)$-type, defined by spin groups associated to quadratic spaces over a totally real number field of degree at least 6 subject to some natural constraints of signatures.
\end{abstract}

\maketitle


\section{Introduction}\label{section Introduction}
This paper is dedicated to the study of the Oort conjecture and we show that certain Shimura subvarieties  uniformized by Hermitian symmetric domains associated to $\SU(n,1)$ and $\SO(n,2)$ are not contained generically in the Torelli locus. In this section we recall the conjecture, review some related results, and explain the main idea of our work.

\subsection{Torelli locus and the conjecture of  Oort}\label{subsection Torelli locus and the conjecture of Oort}
We first recall the Oort conjecture for the Torelli locus. A more thorough survey of the subject is found in \cite{mo11}.

Fix $\ell\geq 3$ an integer, we have  $\calm_g=\calm_{g,\ell}$ the moduli space of smooth projective curves over  $\mathbb C$ of genus $g\geq 2$ with a full level $\ell$-structure, and  $\cala_g=\cala_{g,\ell}$  the moduli space of $g$-dimensional principally polarized  abelian varieties over $\mathbb C$ with a full level-$\ell$ structure.
In this paper we treat them as the moduli schemes over $\mathbb C$ of the corresponding moduli functors. No specific choice of the level $\ell(\geq3)$ is made because it is only imposed to assure the representability, which plays no essential role in our study.

The Torelli morphism $j^\circ:\Mcal_g\ra\Acal_g$ sends a curve $C$ in $\Mcal_g$ to its Jacobian $\Jac(C)$ endowed with its canonical polarization. We write $\Tcal_g^\circ$ for the image of $j^\circ$, and $\Tcal_g$ for the closure of $\Tcal_g^\circ$ in $\Acal_g$; $\Tcal_g$ is called the Torelli locus. It is known that $\Tcal_g^\circ$ is open in $\Tcal_g$, and is referred to as the open Torelli locus. We also have the hyperelliptic Torelli locus $\Tcal\Hcal_g$ in $\Tcal_g$, corresponding to Jacobians of hyperelliptic curves. A closed subvariety $Z\subseteq\cala_g$ of positive dimension is said to be contained {\emph generically} in the Torelli locus, if $Z\subseteq\calt_g$ and $Z\cap\calt_g^\circ\neq\emptyset$ both hold.

The moduli scheme $\Acal_g=\Acal_{g,\ell}$ actually admits a geometrically connected model over
$\Qbb(\zeta_\ell)$ with $\zeta_\ell$ being a primitive $\ell$-th root of 1, cf.\,\cite{mumfordGIT}.
As is explained in \cite{hida shimura,milne05,moonen98}, the moduli scheme $\Acal_g=\Acal_{g,\ell}$ is {  isomorphic to} a connected Shimura variety,
namely a geometrically connected component of the Shimura variety defined by the Shimura datum $(\GSp_{2g},\Hscr_g^\pm)$ associated to the $\Qbb$-group of symplectic similitude
$\GSp_{2g}$, using some compact open subgroup $K(\ell)\subset \GSp_{2g}(\adele)$, cf. Example \ref{the Siegel modular varieties and its Shimura subvarieties}.
In $\cala_{g}$ there are Shimura subvarieties (cf. Definition \ref{Shimura subvarieties}), which are moduli subspaces classifying abelian varieties with prescribed Hodge classes.
In the literature they are also called special subvarieties or Shimura subvarieties of Hodge type. These subvarieties are totally geodesic,
and they are locally symmetric, i.e., uniformized by Hermitian symmetric domains equivariantly embedded in $\Hscr_g^\pm$.
A more natural description of Shimura subvarieties involves the language of Shimura subdata,
which will be presented later in Section \ref{subsection Shimura varieties and Shimura subvarieties}.
In particular,  the Shimura subvarieties of dimension zero in $\Acal_g$ are exactly the CM points, i.e., points parameterizing CM abelian variety. Any Shimura subvariety contains a Zariski dense subset of CM points.

When $g=1,2,3$, the Torelli morphism is   of dense image. It was conjectured by Coleman \cite{coleman87} that when the genus $g$ is sufficiently large,
there should be at most finitely many CM points   contained in the open Torelli locus $\calt_g^\circ$.
Oort \cite{oort97} made the following conjecture by combining Coleman's conjecture with the conjecture of Andr\'e-Oort.

\begin{conjecture}[(Oort)]\label{Oort conjecture}
For $g$ large, there does not exist any Shimura subvarieties of strictly positive dimensions contained generically in the Torelli locus $\calt_g$.
\end{conjecture}

The Andr\'e-Oort conjecture predicts that a closed geometrically irreducible subvariety  in a Shimura variety is a Shimura subvariety if and only if
it contains a Zariski dense subset of CM points. It is thus immediate that the formulation of Oort is equivalent to
the one by Coleman modulo Andr\'e-Oort. The readers are referred to \cite{noot06,scanlon bourbaki}, etc. for surveys on  recent progress towards the Andr\'e-Oort conjecture.

Although Coleman made his conjecture for $g\geq 4$, counterexamples { have been} found for $4\leq g\leq 7$ (cf. \cite{djongnoot89,frghpe14,mohajerzuo,moonen10}).
However, if one aims at the non-existence of Shimura subvarieties of a certain ``type'' specified by some algebraic constructions with $g$ sufficently large,
then much more evidence is known. For instance, Dwork-Ogus \cite{do86},
de Jong-Noot \cite{djongnoot89} and Moonen-Oort \cite{moonen10,mo11} studied Shimura subvarieties arising from cyclic covers of $\bbp^1$ using $p$-adic Hodge theory.
Kukulies \cite{kukulies10} studied rational Shimura curves with strictly maximal Higgs field using Viehweg-Zuo's characterization and the Sato-Tate conjecture for modular curves.
In \cite{luzuo14} the second and the third named authors studied Shimura curves of Mumford type using the Arakelov inequality of Higgs bundles.

In \cite{hain99} Hain studied the Oort conjecture using properties of mapping class groups to rule out locally symmetric subvarieties in $\Acal_g$ of high ranks.
In fact, he proves that if a Shimura subvariety $M$ of $\Acal_g$ admits no locally symmetric divisors and is contained generically in the Torelli locus,
then it is
\begin{enumerate}
\item either a ball quaotient, i.e., uniformized by the complex $n$-ball $\mathbb{B}^n$ which is the Hermitian symmetric domain associated to $\SU(n,1)$;

\item or $g\geq 3$ and each component of $M^\mathrm{dec}$, the locus corresponding to Jacobians of singular curves, is of codimension at least 2 in $M$,
      the intersection $(M-M^{\mathrm{dec}})\cap\Tcal\Hcal_g$ is non-empty of codimension 1, and the family of Jacobians does not lift to a family of curves over $M$.
\end{enumerate}
Inspired by Hain's results, de Jong and Zhang \cite{djongzh07} proved the non-existence of Hilbert modular varieties in the Torelli locus $\Tcal_g$ for $g>4$.

In the works of Kudla, etc. cf. \cite{kudla orthogonal} and \cite{kudla rapoport}, one encounters Shimura subvarieties of $\SU(n,1)$-type and of $\SO(n,2)$-type,
cf. Definitions \ref{Shimura data of SU(n,1)-type} and \ref{Shimura data of SO(n,2)-type},
which contain Shimura subvarieties in each codimension. Hain's treatment does not apply to these Shimura subvarieties.
But they contain Shimura curves, and they become the main object of our study based on slope properties of Higgs bundles on curves.

\subsection{The main results and the idea of the proofs}
The decomposition of canonical Higgs bundles on Shimura subvarieties, especially on Shimura curves, plays an essential role in this paper, cf. Section \ref{subsection Higgs bundles on Shimura varieties}.
Roughly speaking, for $C\subset\Acal_g$ a smooth closed curve, the canonical Higgs bundle $\Ecal_C$ on $C$ is the Hodge bundle on $C$ given by the universal family of abelian varieties restricted  over $C$.
From the Simpson correspondence it follows that this Higgs bundle is completely determined by the complex representation of the fundamental group $\pi_1(C)$ on the De Rham cohomology of the abelian variety.
The canonical Higgs bundle thus decomposes into $\Ecal_C=\Fcal_C\oplus\Ucal_C$, with $\Ucal_C$ the maximal unitary Higgs subbundle corresponding to the maximal  subrepresentation
on which $\pi_1(C)$ acts through a compact unitary group.
These Higgs bundles are induced by the universal family of abelian varieties,
and thus they are of the form $\Lambda_C=\Lambda_C^{-1,0}\oplus \Lambda_C^{0,-1}$ with $\Lambda\in\{\Ecal,\Fcal,\Ucal\}$ following Hodge decomposition of abelian varieties
(here we follow \cite{deligne pspm} and \cite{milne05} for the convention on Hodge types, where our $H^{p,q}$ is recognized as $H^{-p,-q}$ in complex geometry).
This decomposition extends to the logarithmic Higgs bundle on a smooth compactification $\Cbar$ of $C$ by joining a divisor $\partial C=\Cbar - C$,
i.e., $\Ecal_\Cbar=\Fcal_\Cbar\oplus\Ucal_\Cbar$ with its $(\minuszero)$-part being $\Ecal_\Cbar^\minuszero=\Fcal_\Cbar^\minuszero\oplus\Ucal_\Cbar^\minuszero$.
Viehweg and the last named author \cite{vz04} have obtained a numerical control on the slope of $\Fcal_C^\minuszero$:

\begin{theorem}[(Viehweg-Zuo)]\label{Arakelov (in)equality}
Let $C\subset\Acal_g$ be a smooth closed curve, such that in some smooth compactification $\Acalbar_g$ of $\Acal_g$,
the closure $\Cbar$ of $C$ is a smooth projective curve. Then in the decomposition of the logarithmic Higgs bundle
\begin{equation}\label{eqndecompositionE_C}
\Ecal_\Cbar=\Fcal_\Cbar\oplus\Ucal_\Cbar
\end{equation}
with $\Fcal_\Cbar=\Fcal_\Cbar^\minuszero\oplus \Fcal_\Cbar^\zerominus$
and $\Ucal_\Cbar$ the maximal unitary Higgs subbundle, we have the following Arakelov inequality 
\begin{equation}\label{eqnarakelovineqE_C}
\deg(\Fcal_\Cbar^\minuszero)\leq \frac{\rank\Fcal_\Cbar^{-1,0}}{2} \cdot \deg\Omega^1_\Cbar(\log\partial C).
\end{equation}
Moreover, if $C$ is a Shimura curve, then the equality holds in \eqref{eqnarakelovineqE_C} and $\Fcal_\Cbar^\minuszero$
is poly-stable, i.e., it is a direct sum of stable bundles with the same slope.
\end{theorem}

On the other hand, if a  Shimura curve $C$ is contained generically in the Torelli locus $\Tcal_g$, then we can construct a commutative diagram
$$\xymatrix{B\ar[r]^{j_B}\ar[d] & C\ar[d]^\cap\\ \Mcal_g\ar[r]^{j^\circ} & \Acal_g}$$
where $B$ is the normalization of the pull-back $(j^\circ)^\inv C$ in $\Mcal_g$.
The morphism $B\ra \Mcal_g$ gives rise to a relative $B$-curve $f:S\ra B$, which is a surface fibred over $B$ by curves of genus $g$.
It admits a compactification into a fibration $\fbar:\Sbar\ra\Bbar$ of semi-stable curves over the smooth compactification $\Bbar$ of $B$,
and the morphism $j_B$ extends naturally to $j_{\ol B}: \Bbar \to \Cbar$.
On $\Bbar$ we have the  sheaf $\fbar_*\omega_{\Sbar/\Bbar}$ where $\omega_{\Sbar/\Bbar}=\omega_{\Sbar}\otimes\fbar^*\omega_\Bbar^{\vee}$ is the relative canonical sheaf.
It is a locally free sheaf on $\Bbar$ of rank $g$, and admits a decomposition
$\fbar_*\omega_{\Sbar/\Bbar}=\Fcal^{-1,0}_\Bbar\oplus\Ucal_\Bbar^{-1,0}$ with $\Fcal_\Bbar^{-1,0}=j_\Bbar^*\big(\Fcal_\Cbar^{-1,0}\big)$ semi-stable and
$\Ucal_\Bbar^{-1,0}=j_\Bbar^*\big(\Ucal_\Cbar^{-1,0}\big)$ unitary (cf. \cite[\S\,3]{luzuo14}).
Using Xiao's technique on the slope inequality of fibred surfaces,
we derive an upper bound for $\rank\Ucal_\Bbar^\minuszero$, hence the following theorem
\begin{theorem}[(Exclusion of Shimura curves)]\label{exclusion of Shimura curves}
Let $C\subset \Acal_g$ be a curve with the Higgs bundle decomposition $\Ecal_C=\Fcal_C\oplus\Ucal_C$, where $\Ucal_C$ is the maximal unitary Higgs subbundle.
If
\begin{equation}\label{eqnassumption in exclusion of curves}
\rank\Ucal_C^{-1,0}> (5g+1)/6, \quad\text{or equivalently}\quad \rank\Fcal_C^\minuszero<(g-1)/6,
\end{equation}
then $C$ is not contained generically in the Torelli locus. If moreover $C$ is a Shimura curve, then it suffices to require $\rank\Ucal_C^\minuszero\geq(4g+2)/5$, or equivalently, $\rank\Fcal_C^{\minuszero}\leq(g-2)/5$.
\end{theorem}

\begin{remark}
According to the above result together with \cite[Theorem A]{luzuo14}, we shall say that the Oort conjecture for Shimura curves of the extremal cases have been solved:
the case $\Ecal_\Cbar=\Fcal_\Cbar$, i.e., $\Ucal_\Cbar=0$, also called curves with strictly maximal Higgs field, has been dealt with in \cite{luzuo14};
and the opposite case, namely the canonical Higgs bundle having a large unitary Higgs subbundle, is treated in the above theorem.
However the techniques involved are totally different.
For a Shimura curve $C$ contained generically in the Torelli locus,
in \cite{luzuo14} we excluded $C$ by establishing a strict Arakelov inequality contradiction to the equality in \eqref{eqnarakelovineqE_C}
under the assumption $\Ecal_\Cbar=\Fcal_\Cbar$ using Miyaoka-Yau's theorem and Moriwaki's slope inequality on fibred surfaces;
here we obtain a bound on the rank of the unitary Higgs subbundle
by using the poly-stability of $\Fcal_\Cbar$ and  Xiao's technique of the slope inequality for fibred surfaces.
\end{remark}

Our main theorems focus on Shimura subvarieties in $\Acal_g$ that
contain Shimura curves. If $M$ is such a Shimura variety that
contains a Zariski dense subset of Shimura curves satisfying the
condition in Theorem \ref{exclusion of Shimura curves}, then $M$ is
not contained generically in $\Tcal_g$. The Shimura subvarieties we
work with are either of $\SU(n,1)$-type or of $\SO(n,2)$-type,
namely they are defined by Shimura subdata
$(\Gbf,X)\subset(\GSp_{2g},\Hscr_g^\pm)$ of the following forms:
\begin{list}{}
{\setlength{\labelwidth}{1.8cm}
\setlength{\leftmargin}{1.9cm}}
\item[$\SU(n,1)$:] $\Gbf^\der=\Res_{F/\Qbb}\SU(H)$, where $h:H\times H\ra E$ is a non-degenerate Hermitian space over a CM field $E$ with totally real subfield $F$, such that  $h$ is  of signature $(n,1)$ along one single embedding $\tau:F\mono\Rbb$, and is definite along the other embeddings;
\item[$\SO(n,2)$:] $\Gbf^\der=\Res_{F/\Qbb}\Spin(H)$, where $Q:H\times H\ra F$ is a non-degenerate quadratic space over a totally real field $F$, such that $Q$ is of signature $(n,2)$ along one single embedding $\tau:F\mono\Rbb$, and definite along the other embeddings.
\end{list}
Following \cite{kudla orthogonal} and \cite{kudla rapoport}, we know
that in each case the Shimura subvariety $M$  defined by $(\Gbf,X)$
contains Shimura curves, in fact Shimura subvarieties in each
codimension.

The homomorphisms $\Gbf^\der\mono\Sp_{2g}$ are rational symplectic representations in the sense of \cite{satake classification},
and they characterize many geometric properties of the embeddings $X\mono\Hscr_g^\pm$.
On the other hand, the canonical Higgs bundles $\Ecal_M$ on the Shimura subvariety $M$ is determined by the representation of $\pi_1(M)$,
which is determined by $\Gbf^\der\ra\Sp_{2g}$ because $\pi_1(M)$ is essentially a congruence subgroup in $\Gbf^\der(\Qbb)^+$.
We can thus compute the decomposition of $\Ecal_M$ using the symplectic representations, and our main results are:

\begin{theorem}[(Unitary case)]\label{Main theorem}
Let $M\mono \Acal_g$ be a Shimura subvariety of $\SU(n,1)$-type ($n\geq 1$), with $\Gbf^\der=\Res_{F/\Qbb}\SU_V$ for some Hermitian space $H$
over a CM extension $E/F$ subject to the constraints of signatures as above.
Assume that the representation $\Gbf^\der\ra\Sp_{2g}$ is the scalar restriction from $F$ to $\Qbb$ of some $\tau$-primary symplectic representation
$\SU(H)\ra\Sp_L$ of type $\Lambda_m$: $L\otimes_{F,\tau}\Rbb\isom\Lambda_m^{\oplus N}$ ($1\leq m\leq n$).
If
\begin{equation}\label{eqnassumption in main theorem}
\left(1- \frac{10m(n-m+1)}{dn(n+1)}\right)\cdot g\geq 2,\qquad \text{where~}d=[F:\Qbb],
\end{equation}
then $M$ is not contained generically in the Torelli locus $\Tcal_g$.
\end{theorem}

Here the notion of $\tau$-primary representation of type $\Lambda_m$ (cf. Definition \ref{primary type}) arises naturally from the classification of Satake \cite{satake summary}
(which summarizes \cite{satake embedding} and \cite{satake classification}),
and we do not need finer information over $F$ or over $\Qbb$. The key point is to restrict the symplectic representation $\Gbf^\der\ra\Sp_{2g}$ to $\Gbf'^\der$ where $(\Gbf',X')$ is the Shimura subdatum defining some Shimura curve $C$ in $M$,
and compare the rank of unitary part $\Ucal_C$ of $\Ecal_C$ with the inequality \eqref{eqnassumption in exclusion of curves} in Theorem \ref{exclusion of Shimura curves}.
In fact we do obtain a more general inequality for the rank   $\Ucal_C$ for the restriction of a general symplectic representation $\Gbf^\der\ra\Sp_{2g}$,
cf. Corollaries \ref{formula of SU(n,1)-type} and \ref{formula of SU(1,1)-type},
which is an immediate consequence of the inequality above with various $\Lambda_m$ taken into consideration.

Note that under the assumption of the theorem we actually have $g=Nd\binom{n+1}{m}$ greater than $d\binom{n+1}{m}$, hence the inequality admits some relaxed forms which are more convenient:

\begin{corollary}[(Unitary case)]\label{Main corollary}
Under the assumptions and notations of Theorem \ref{Main theorem}, the Shimura subvariety of $\SU(n,1)$-type is not contained generically in the Torelli locus $\calt_g$ if:
\begin{itemize}
\item $d(n+1)\geq 12$ when the symplectic representation is $\tau$-primary of type $\Lambda_1$;

\item $d\geq 3$ and $n\geq 6$ or $d\geq 4$ with $n$ arbitrary when the symplectic representation is $\tau$-primary of type $\Lambda_m$ for general $m\geq 2$.
\end{itemize}
\end{corollary}

In the orthogonal case, the Shimura subvarieties of $\SO(n,2)$-type are defined by spin groups,
and the symplectic representations are always $\tau$-primary of spinor type when no trivial subrepresentation is allowed.
Furthermore, when restricted to a Shimura curve $C$, the maximal unitary part $\Ucal_C$ of $\Ecal_C$ is completely determined by the contribution from the embeddings
$F\mono\Rbb$ along which the quadratic spaces are definite.
Hence the conclusion is simpler:
\begin{theorem}[(Orthogonal case)]\label{main theorem in the orthogonal case} Let $M\mono\Acal_g$ be a Shimura subvariety of $\SO(n,2)$-type $(n\geq 1)$,
with $\Gbf^\der=\Res_{F/\Qbb}\Spin(H)$ for some quadratic space $H$ over a totally real field $F$ subject to the natural constraints of signatures as above.
Then $M$ is not contained generically in $\Tcal_g$ if $d=[F:\Qbb]\geq 6$.
\end{theorem}

\subsection{Organization of the paper}\label{subsection Organization of the paper}
In Section \ref{seection Preliminaries}, we recall the basic notions of Shimura subvarieties in $\Acal_g$ and the decomposition of Higgs bundles on them.
In Section \ref{section Xiao's technique for curves}, we prove the bound of the rank of the unitary part in the Higgs bundle associated to a family of semi-stable curves, and obtain the numerical criterion to exclude Shimura curves from the Torelli locus.
In Section \ref{section Shimura varieties of unitary and orthogonal types}, we explain the construction of Shimura data of unitary type and orthogonal type associated to Hermitian spaces and quadratic spaces respectively, and specify to the case of $\SU(n,1)$-type and $\SO(n,2)$-type.
In Section \ref{section Symplectic representations in the unitary case}, we collect facts from Satake's classification of symplectic representations of semi-simple groups of $\SU(n,1)$-type, and we compute the rank of the unitary part of the Higgs bundle on  Shimura curves embedded in Shimura subvarieties of $\SU(n,1)$-type, which leads to the proofs of the main results in the unitary case.
Finally, in Section \ref{section Symplectic representations in the orthogonal case}, we recall the construction of spinor representations and Satake's classification for spin groups, and we compute the rank of the unitary part which ends the proof in the orthogonal case.

\subsection*{Acknowledgements} This work is supported by SFB/Transregio 45 Periods, Moduli Spaces and Arithmetic of Algebraic Varieties of the DFG (Deutsche Forschungsgemeinschaft),
and also supported by National Key Basic Research Program of China (Grant No. 2013CB834202). The first named author is partially supported by National Natural Science Foundation of China (Grant No. 11301495).

During the preparation of this work, we  learned the sudden passing away of Professor Gang Xiao. Prof. Xiao has been a friend and a guide to mathematics for many years, with great contributions to and impacts on the research of algebraic geometry in China. We would like to express our gratitude, grief, and memory for Prof. Xiao through this paper, where his slope inequality plays a crucial role.

\subsection*{Notations and conventions}\label{subsection Notations and conventions}
Denote by $\Sbb$ the Deligne torus $\Res_{\Cbb/\Rbb}\Gbb_{\mrm,\Rbb}$,   $\ibf$ a fixed square root of -1 in $\Cbb$, $\adele$ the ring of finite adeles of $\Qbb$, and $\Qac$ the algebraic closure of $\Qbb$ inside $\Cbb$. A $\Qbb$-group $\Gbf$ is compact if the Lie group given by $\Gbf(\Rbb)$ is compact.

For $k$ a field, by linear $k$-group we mean affine algebraic group $k$-schemes, among which we have reductive $k$-groups, semi-simple $k$-groups, $k$-tori, etc. defined in the standard way.
If $\Gbf$ is a reductive $\Qbb$-group, we write $\Gbf^\circ$ for the neutral component of $\Gbf$ for the Zariski topology,
$\Gbf(\Rbb)^+$ for the neutral component of the Lie group $\Gbf(\Rbb)$ for the analytic topology (i.e., the one on the manifold $\Gbf(\Rbb)$ locally given by the archimedean metric),
and $\Gbf(\Rbb)_+$ for the preimage of $\Gbf^\ad(\Rbb)^+$ \wrt the homomorphism $\Gbf(\Rbb)\ra\Gbf^\ad(\Rbb)$.
We also write $\Gbf(\Qbb)^+$ resp. $\Gbf(\Qbb)_+$ for the intersection $\Gbf(\Qbb)\cap\Gbf(\Rbb)^+$ resp. $\Gbf(\Qbb)\cap\Gbf(\Rbb)_+$.
We write $\Xfrak(\Gbf)$ for the set of $\Rbb$-group homomorphisms $\Hom(\Sbb,\Gbf_\Rbb)$, on which the Lie group $\Gbf(\Rbb)$ acts from the left by conjugation.

\section{Preliminaries}\label{seection Preliminaries}
\subsection{Shimura varieties and Shimura subvarieties}\label{subsection Shimura varieties and Shimura subvarieties}
We first recall the classical definitions of Shimura data and Shimura varieties given in \cite{deligne pspm}.

\begin{definition}[(Shimura data)]\label{Shimura data}
(1) A Shimura datum is a pair $(\Gbf,X)$, where $\Gbf$ is a connected reductive $\Qbb$-group,
whose adjoint quotient $\Gbf^\ad$ admits no compact $\Qbb$-factors, and $X\subset\Xfrak(\Gbf)$ is a single $\Gbf(\Rbb)$-orbit, such that for any $x\in X$ we have
\begin{enumerate}
\item[SD1.] the composition $\Ad_\Gbf\circ x:\Sbb\ra\Gbf_\Rbb\ot{\Ad_\Gbf}\ra\GL_{\gfrak,\Rbb}$ defines a rational pure Hodge structure of type $\{(-1,1),(0,0),(1,-1)\}$ on the Lie algebra $\gfrak=\Lie\Gbf$;
\item[SD2.] the conjugation by $x(\ibf)$ induces a Cartan involution on the Lie group $\Gbf^\ad(\Rbb)$.
\end{enumerate}

Under these constraints, $X$ is a complex manifold on which $\Gbf(\Rbb)$ acts by holomorphic automorphisms.
The center of $\Gbf(\Rbb)$ acts on $X$ trivially, and each connected component $X^+$ of $X$ is the Hermitian symmetric domain
(i.e., of non-compact type) associated to the semi-simple Lie group $\Gbf^\ad(\Rbb)^+$.

(2) Let $(\Gbf,X)$ be a Shimura datum. A Shimura subdatum of $(\Gbf,X)$ is a Shimura datum $(\Gbf',X')$ such that \begin{itemize}
\item $\Gbf'$ is a $\Qbb$-subgroup of $\Gbf$, and $X'$ is a subset of $X$;
\item the inclusion $X'\mono X$ is equivariant \wrt the Lie group homomorphism $\Gbf'(\Rbb)\mono\Gbf(\Rbb)$.
\end{itemize}

It turns out that $X'$ is a subset of $\Xfrak(\Gbf)$ consisting of points $x:\Sbb\ra\Gbf_\Rbb$ that have their image in $\Gbf'_\Rbb$.
When $X'$ is zero dimensional, $\Gbf'$ has to be a $\Qbb$-torus, and $X'$ is a single point.
In this case $(\Gbf',X')$ is said to be a CM subdatum of $(\Gbf,X)$, motivated from the notion of CM abelian varieties, cf. Example \ref{the Siegel modular varieties and its Shimura subvarieties}.
\end{definition}

\begin{definition}[(Shimura varieties)]\label{Shimura varieties}
Let $(\Gbf,X)$ be a Shimura datum, and let $K\subset\Gbf(\adele)$ be a \cosg.
The Shimura variety associated to $(\Gbf,X)$ at level $K$ is a quasi-projective algebraic variety $M_K(\Gbf,X)$ over $\Qac$,
whose complex points are described as $$M_K(\Gbf,X)(\Cbb)=\Gbf(\Qbb)\bsh \big(X\times\Gbf(\adele)/K\big)$$
where $\Gbf(\Qbb)$ acts on the product $X\times\Gbf(\adele)/K$ from the diagonal.
Fix $X^+$ a connected component of $X$ and let $\Sigma$ be a set of representatives in $\Gbf(\adele)$ of the finite double quotient $\Gbf(\Qbb)_+\bsh\Gbf(\adele)/K$,
then we have $$M_K(\Gbf,X)\isom\Gbf(\Qbb)_+\bsh \big(X^+\times\Gbf(\adele)/K\big)\isom\coprod_{g\in\Sigma}\Gamma_K(g)\bsh X^+$$
where $\Gamma_K(g)=gKg^\inv\cap\Gbf(\Qbb)_+$ is a congruence subgroup of $\Gbf(\Qbb)_+$, which acts on $X^+$ through its image in $\Gbf^\ad(\Qbb)^+$.

Baily and Borel \cite{baily borel} have shown that the quotients $\Gamma_K(g)\bsh X^+$ are normal quasi-projective algebraic varieties over $\Cbb$.
Deligne, Milne, and Borovoi, etc. have shown that the double quotient $\Gbf(\Qbb)\bsh\big(X\times\Gbf(\adele)/K\big)$ admits a unique canonical model over the reflex field $E(\Gbf,X)$,
which is a number field embedded in $\Cbb$, and each connected component $\Gamma_K(g)\bsh X^+$ is defined over a finite abelian extension of $E(\Gbf,X)$, cf. \cite{milne05}.
In our study it suffices to treat Shimura varieties as complex algebraic varieties.
\end{definition}

We will mainly work with connected Shimura data and connected Shimura varieties, in which setting the notion of Shimura subvarieties replace the special subvarieties in the sense of \cite{mo11}.

\begin{definition}[(Connected Shimura data and varieties)]\label{connected Shimura data and varieties}
A connected Shimura datum is a triple $(\Gbf,X;X^+)$ with $(\Gbf,X)$ a Shimura datum and $X^+$ a connected component of $X$.

A connected Shimura variety defined by the connected Shimura datum $(\Gbf,X;X^+)$ is a space of the form $M=\Gamma\bsh X^+$,
where $\Gamma$ is a congruence subgroup of $\Gbf^\der(\Qbb)^+$. Since $\Gbf^\der(\Rbb)^+$ acts on $X^+$ transitively, it follows from the theorem of Baily and Borel that $\Gamma\bsh X^+$ is a complex algebraic variety.

For simplicity we assume that in $\Gamma\bsh X^+$ the congruence subgroup $\Gamma$ is taken to be small enough so that it is free of torsion
and that the homomorphism $\Gamma\ra\Gbf^\ad(\Qbb)$ is injective. This is always possible because the center of $\Gbf^\der(\Qbb)$ is finite.
Thus the action of $\Gamma$ on $X^+$ is faithful and it is isomorphic to the topological fundamental group of $M$. We write $\wp_\Gamma:X^+\ra \Gamma\bsh X^+$ for the uniformization map $x\mapsto \Gamma x$.

\end{definition}

\begin{definition}[(Shimura subvarieties)]\label{Shimura subvarieties} Let $M=\Gamma\bsh X^+$ be a connected Shimura variety,
defined by $(\Gbf,X;X^+)$ and some congruence subgroup $\Gamma\subset\Gbf^\der(\Qbb)^+$, with $\wp_\Gamma$ the uniformization map.

A connected Shimura subdatum of $(\Gbf,X;X^+)$ is a triple of the form $(\Gbf',X';X'^+)$, where $(\Gbf',X')$ is a Shimura subdatum of $(\Gbf,X)$,
and $X'^+$ is a connected component of $X'$ which is contained in $X^+$. The Shimura subvariety of $M$ associated to the subdatum $(\Gbf',X';X'^+)$ is $M'=\wp_\Gamma(X'^+)$.


When $(\Gbf',X')$ is a CM subdatum, the Shimura subvariety is a single point, and we call it the CM point in $M$ associated to $(\Gbf',X';X'^+)$ (in this case $X'=X'^+$ consists of a single point).

\end{definition}
Note that Shimura subvarieties in $\Gamma\bsh X^+$ can be equivalently characterized as totally geodesic subvarieties containing CM points, due to \cite{moonen98}.

\begin{remark}
In \cite{milne05} the notion of connected Shimura varieties is defined as quotients of the form $\Gamma\bsh X^+$,
where $X^+$ comes from some connected Shimura datum $(\Gbf,X;X^+)$, and $\Gamma$ is a congruence subgroup in $\Gbf^\ad(\Qbb)^+$.
Choose $\Gamma$ to be a congruence subgroup of $\Gbf^\der(\Qbb)^+$ and write $\Gamma'$ for its image in $\Gbf^\ad(\Qbb)^+$,
then the natural projection $f:\Gamma\bsh X^+\ra\Gamma'\bsh X^+$ is a finite morphism of complex algebraic varieties.
Moreover, let $Z\subset \Gamma\bsh X^+$ be a geometrically irreducible closed subvariety,
then $Z$ is a Shimura subvariety \ifof $f(Z)$ is a Shimura subvariety in $\Gamma'\bsh X^+$,
because $Z$ is a Shimura subvariety \ifof one geometrically irreducible component of $\wp_\Gamma^\inv(Z)$ in $X^+$ is $X_1^+$
coming from some connected Shimura subdatum $(\Gbf_1,X_1;X_1^+)$, and this is equivalent to $f(Z)\subset\Gamma'\bsh X^+$ being a Shimura subvariety.

\end{remark}

We also have the following
\begin{lemma}[(Uniformization of Shimura subvarieties)]\label{uniformization of Shimura subvarieties}(1) Let $k$ be a field, and $\Gbf$ be a connected reductive $k$-group, with $\Hbf\subset\Gbf$ a connected reductive $k$-subgroup. Then the normalizer $\Nbf=\Nbf_\Gbf\Hbf$ in $\Gbf$ is reductive, and the $k$-subgroup generated by $\Hbf$ and its centralizer $\Zbf=\Zbf_\Gbf\Hbf$ is of finite index in $\Gbf$.

(2) Let $M'\subset M$ be a Shimura subvariety defined by some connected Shimura subdatum $(\Gbf',X';X'^+)\subset(\Gbf,X;X^+)$ using some torsion-free congruence subgroup $\Gamma$, i.e., $M=\Gamma\bsh X^+$ and $M'=\wp_\Gamma(X'^+)$. Take $\Gamma'=\Gamma\cap\Gbf'^\der(\Qbb)^+$, then the evident map $\Gamma'\bsh X'^+\ra\wp_\Gamma(X'^+)$ is a finite \'etale covering.

\end{lemma}

\begin{proof}
(1) The main arguments below are reproduced from the proof by J. Humphreys in \cite{humphreys mathoverflow}, as we have not yet found explicit references for this seemingly well-known fact.

We may assume that $k$ is algebraically closed. If $\Hbf$ is semi-simple, then it is isogeneous to a finite direct product of simple $k$-groups. The natural $k$-group homomorphism $\Nbf\ra\Aut_k(\Hbf)$ sending $n$ to the conjugation by $n$ has the centralizer $\Zbf=\Zbf_\Gbf\Hbf$ as its kernel, and its image contains the inner automorphisms given by $\Hbf$ itself. Hence its image is of finite index in $\Aut_k(\Hbf)$ because $\Aut_k(\Hbf)$ only differs from $\Hbf^\ad$ by a finite group of outer automorphisms: when $\Hbf$ is simple, this follows from the finiteness of the automorphism group of the Dynkin diagram; when $\Hbf$ is semi-simple, it suffices to join the finite permutations among simple factors of the same types. Finally, when $\Hbf$ is reductive, it is isogeneous to the direct product of a semi-simple $k$-group with a $k$-torus, and it suffices to argue by the rigidity of torus, i.e., the centralizer of a torus is of finite index in its normalizer, cf. \cite[\S\,16.3]{humphreys linear algebraic groups}.

(2) Recall that the Mumford-Tate group $\MT(S)$ of a subset $S\subset X$ is the smallest $\Qbb$-subgroup $\Hbf$ such that $\Hbf_\Rbb\supset x(\Sbb)$ for all $x\in S$. It is clear from \cite{andre fixed part} that $\Gbf'^\der=\Hbf^\der$ for $\Hbf=\MT(X'^+)$. For $g\in\Gbf^\der(\Rbb)^+$ and $x'\in X'^+$, $g(x')\in X'^+$ implies that $gx'(\Sbb)g^\inv\subset\Hbf_\Rbb$, and $x'$ running through $X'^+$ gives $g$ normalizing $\Hbf_\Rbb$ i.e., $g\in\Nbf_\Gbf\Hbf(\Rbb)$. Conversely, if $g\in\Nbf_{\Gbf^\der}\Hbf(\Rbb)^+$, then the conjugation by $g$ leaves the $\Gbf'$ stable, hence it stabilizes $X'$, and it stabilizes $X'^+$ because $g$ lies in the same path-connected component as the neutral element. Here the notation $\Nbf_{\Gbf^\der}\Hbf$ makes sense because both $\Gbf^\der$ and $\Hbf$ are $\Qbb$-subgroups of $\Gbf$.

We thus put $\Nbf=\Nbf_{\Gbf^\der}\Hbf=\Nbf_{\Gbf}\MT(X'^+)$, and we have just proved that $\Nbf(\Rbb)^+$ is of finite index in the normalizer $N$ of $X'^+$ in $\Gbf^\der(\Rbb)^+$. From (a) we know further that $\Nbf$ is isogeneous to a product $\Zbf\times\Hbf^\der=\Zbf\times\Gbf'^\der$, where $\Zbf$ is generated by the centralizer of $\Hbf$ and the center of $\Hbf$. Note that $\Zbf(\Rbb)$ is actually compact, because conjugation by $g\in\Zbf(\Rbb)$ commutes with any $x'(\ibf)$ for $x'\in X'^+$, and the centralizer of $x'(\ibf)$ in $\Gbf(\Rbb)$ is compact modulo the center because it induces a Cartan involution on $\Gbf^\ad_\Rbb$. We may thus shrink $\Gamma$ so that $\Gamma\cap N=\Gamma\cap\Nbf(\Qbb)^+$, which is further equal to $\Gamma\cap\Hbf(\Qbb)^+$ because $\Gamma\cap\Zbf(\Qbb)^+$ is trivial for $\Gamma$ small enough. Hence the action of $\Gamma\cap\Nbf(\Qbb)^+$ on $X'^+$ only differs from the natural action of $\Gamma\cap\Gbf'^\der(\Qbb)^+$ by a finite quotient, which is \'etale because we have only worked with torsion-free congruence subgroups.
\end{proof}

\begin{example}[(Siegel modular varieties and its Shimura subvarieties)]\label{the Siegel modular varieties and its Shimura subvarieties}
Let $(V,\psi)$ be a finite-dimensional symplectic vector space over $\Qbb$ of dimension $2g$. We have the pair $(\GSp_V,\Hscr_V^\pm)$ where \begin{itemize}
\item $\GSp_V$ is the $\Qbb$-group of symplectic similitude of $(V,\Qbb)$;
\item $\Hscr_V^\pm$ is the set of polarizations of $(V,\psi)$, i.e., complex structures $h:\Cbb^\times\ra\GL_\Rbb(V_\Rbb)$ such that $(x,y)\in V_\Rbb\times V_\Rbb\mapsto \psi(h(\ibf)x,y)$ is symmetric and definite.
\end{itemize}
Then $\Hscr_V^\pm$ is identified with the Siegel double half space,
which is the $\GSp_V(\Rbb)$-orbit of any homomorphism $x:\Sbb\ra\GL_{V,\Rbb}$ that polarizes $\psi$.
It consists of two connected components $\Hscr_V^+\coprod\Hscr_V^-$ depending on the sign of the definite quadratic form
$\psi(x(\ibf)\_,\_)$. The pair $(\GSp_V,\Hscr_V^\pm)$ is a Shimura datum, and we call it the Siegel datum defined by the symplectic space $(V,\psi)$.

Assume that $(V,\psi)$ comes from an integral symplectic module $(L,\psi_L)$ over $\Zbb$ whose discriminant  equals 1.
Take the $\ell$-th principal level structure $K=K(\ell)=\Ker(\GSp_L(\Zbhat)\ra\GSp_L(\Zbb/\ell))$,
where $\ell\geq 3$ is an integer, we have the Shimura variety $M_K(\GSp_V,\Hscr_V^\pm)$, with its canonical model over $\Qbb$. As a $\Qbb$-scheme,
it is isomorphic to the composition $\Acal_{g,\ell}\ra\Spec\Qbb(\zeta_\ell)\ra\Spec\Qbb$,
where $\Acal_{g,\ell}\ra\Spec\Qbb(\zeta_\ell)$ is the fine moduli scheme over the $\ell$-th cyclotomic field $\Qbb(\zeta_\ell)$
parametrizing principally polarized abelian schemes with full level-$n$ structures.
We may identify $\Acal_{g,\ell}$ with a geometrically connected component of $M_K(\GSp_V,\Hscr_V^\pm)$,
which is a connected Shimura variety isomorphic to $\Gamma(\ell)\bsh \Hscr_V^+$, with $\Gamma(\ell)$ the $\ell$-th principal congruence subgroup $\Gamma(\ell)=\Ker(\Sp_L(\Zbb)\ra\Sp_L(\Zbb/\ell))$.

In the rest of the paper we prefer to write $\Acal_V$ or $\Acal_g$ for the connected Shimura variety defined by $(\GSp_V,\Hscr_V^\pm;\Hscr_V^+)$
with some principal torsion-free congruence subgroup $\Gamma=\Gamma(\ell)\subset\Sp_V(\Qbb)$ ($\ell\geq 3$) given by a suitable integral structure of $(V,\psi)$. 

Shimura subvarieties in $\Acal_V$ are often called Shimura subvarieties of Hodge type, because they are defined by Shimura data of Hodge type,
i.e., Shimura subdata of the Siegel datum $(\Gbf,X)\mono(\GSp_V,\Hscr_V^\pm)$.
Treat the inclusion $\Gbf\mono\GSp_V$ as an algebraic representation $\rho:\Gbf\ra\GSp_V\ra\GL_V$,
then for any $x\in X(\subset\Hscr_V^\pm)$, the composition $\rho\circ x$ defines a complex structure on $V_\Rbb$, hence  a rational Hodge structure of type $\{(-1,0),(0,-1)\}$ on $V$.
The Shimura subvarieties they define in $\Acal_V$ are moduli subspaces parametrizing principally polarized abelian varieties with prescribed Hodge classes, cf. \cite{deligne pspm}.

In particular, when $(\Gbf',X')$ is a CM subdatum, i.e., $\Gbf'$ is a $\Qbb$-torus and $X'$ consists of a single point $x$,
the Shimura subvariety of $M_K(\GSp_V,\Hscr_V^\pm)$ defined by $(\Gbf',X')$ is a CM point, i.e., it corresponds to a CM abelian variety.
\end{example}

\begin{remark}[(Holomorphic equivariant embeddings)]\label{holomorphic equivariant embeddings} If $(\Gbf',X';X'^+)\subset(\Gbf,X;X^+)$ is a connected Shimura subdatum, then the map $X'^+\mono X^+$ is an embedding of complex submanifolds, equivariant \wrt the Lie group homomorphism $\Gbf'^\der(\Rbb)^+\mono\Gbf^\der(\Rbb)^+$. It is an holomorphic equivariant embedding of Hermitian symmetric domains in the sense of \cite{satake embedding}, subject to the further constraint (H2), which we will recall later in Section \ref{section Symplectic representations in the unitary case}, with emphasis on the Siegel case $(\Gbf,X;X^+)=(\GSp_V,\Hscr_V^\pm;\Hscr_V^+)$.

\end{remark}

We end this subsection with the notion of Hecke translation and a reduction lemma for the study of the Oort conjecture.

\begin{definition}[(Hecke translation)]\label{Hecke translation} Let $M=\Gamma\bsh X^+ $ be a Shimura variety defined by $(\Gbf,X;X^+)$, and let $M'\subset M$ be a Shimura subvariety defined by $(\Gbf',X';X'^+)$. For $q\in\Gbf(\Qbb)^+$ we have the Hecke correspondence $$T_q:\Gamma\bsh X^+\ot{a_q}\la \Gamma_q\bsh X^+\ot{b_q}\ra \Gamma\bsh X^+$$ using $\Gamma_q=\Gamma\cap q^\inv\Gamma q$, $a_q(\Gamma_q x)=\Gamma x$, and $b_q(\Gamma_q x)=\Gamma qx$. Both $a_q$ and $b_q$ are finite morphisms of algebraic varieties, and the Hecke translation of a subset $Z\subset M$ by $q$ is understood as the subset $T_q(Z):=b(a^\inv(Z))$, which actually induces a morphism on the cycle groups of $M$. We also define the total Hecke orbit of $Z$ in  $M$ to be the union $\bigcup_{q\in\Gbf(\Qbb)^+}T_q(Z)$.

Note that if $q\in\Gbf(\Qbb)^+$ defines the Hecke correspondence $T_q$, then by the coset decomposition $\Gamma=\coprod_j\Gamma_qa_j$ for finitely many $a_j$'s we see that $T_q(\Gamma x)=\{\Gamma qa_j x\}$. If $M'=\wp_\Gamma(X'^+)$ is a Shimura subvariety defined by $(\Gbf',X';X'^+)$, then $T_q(M')$ is the finite union of Shimura subvarieties given by $(qa_j\Gbf'(qa_j)^\inv,qa_jX';qa_jX'^+)$, and each of these Shimura subvarieties is also called a Hecke translate of $M'$. In particular, Hecke translations respect cycles given by Shimura subvarieties, and the Shimura subdata only differ by rational conjugations.

We also mention  that one may simply use $q\in\Gbf^\der(\Qbb)^+$, because the central elements in $\Gbf(\Qbb)^+$ contribute trivially to the Hecke correspondences.

\end{definition}

It is clear that the definition above is only a rational version of the adelic Hecke translation in \cite{deligne pspm}.

\begin{lemma}[(Reduction to the open Torelli locus)]\label{reduction to the open Torelli locus} Let $M\subset\Acal_V=\Acal_g$ be a Shimura subvariety defined by $(\Gbf,X;X^+)$ contained generically in $\Tcal_g$, and let $M'\subset M$ be a Shimura subvariety of dimension $>0$ defined by $(\Gbf',X';X'^+)$. Then
there exists $g\in \Gbf^\der(\Qbb)^+$ such that the Shimura subvariety $M''$ defined by $(g\Gbf'g^\inv, gX';gX'^+)$ is contained generically in $\Tcal_g$, i.e., $\Tcal_g^\circ\cap M''\neq \emptyset$.

\end{lemma}

\begin{proof}
The starting point is the theorem of real approaximation, cf. \cite[Theorem 5.4]{milne05}: if $\Hbf$ is a linear $\Qbb$-group, then the subset $\Hbf(\Qbb)\subset\Hbf(\Rbb)$ is dense for the archimedean topology on the Lie group $\Hbf(\Rbb)$, and $\Hbf^\der(\Qbb)^+$ is dense in $\Hbf^\der(\Rbb)^+$.  Hence for $M'\subset M$ defined by $(\Gbf',X';X'^+)$, the union of Shimura subvarieties $\bigcup_{q\in\Gbf^\der(\Qbb)^+}\wp_\Gamma(qX'^+)$ defined by $(q\Gbf'q^\inv,qX';qX'^+)$ is dense in $M$, and it equals the total Hecke orbit of $M'$ in $M$, because only elements in $\Gbf^\der(\Qbb)^+$ could contribute non-trivially to the Hecke correspondences. Now that $M\subset\Tcal_g$ meets $\Tcal_g^\circ$ non-trivially and $\Tcal_g^\circ$ is open and Zariski dense in $\Tcal_g$, there is some $q\in\Gbf^\der(\Qbb)^+$ such that $\wp_\Gamma(qX'^+)$ is contained generically in $\Tcal_g$.
\end{proof}

\subsection{Higgs bundles on Shimura varieties}\label{subsection Higgs bundles on Shimura varieties}
We recall a few facts about Higgs bundles associated to variations of Hodge structures and the Simpson correspondence.
We will focus on the case over a connected Shimura variety of Hodge type endowed with a suitable smooth compactification.
Note that we follow \cite{deligne pspm} and \cite{milne05} for the convention on Hodge types, where our $H^{p,q}$ is recognized as $H^{-p,-q}$ in complex geometry.

Let $(\Gbf,X;X^+)$ be a connected Shimura datum. From \cite{deligne pspm} we know that if   $\rho:\Gbf\ra \GL_V$ is an algebraic representation of a single rational weight $n$,
i.e., the composition $\Gbb_{\mrm,\Rbb}\ot{w}\ra\Sbb\ot{x}\ra\Gbf_\Rbb\ot{\rho}\ra\GL_{V,\Rbb}$ is the central cocharacter $t\mapsto t^{-n}\id_{V_\Rbb}$ defined over $\Qbb$,
then the constant sheaf on $X^+$ of stalk $V$ underlies a $\Qbb$-PVHS (polarized variation of rational Hodge structures) of weight $n$, denoted as $\Vbb$.
If we take $V_\Zbb$ an integral structure of $V$ and $\Gamma\subset\Gbf^\der(\Qbb)_+$ a torsion free congruence subgroup stabilizing $V_\Zbb$,
then the $\Zbb$-PVHS $\Vbb_\Zbb$ on $X^+$ with stalk $V_\Zbb$ descends to a $\Zbb$-PVHS $\Vcal$ on the connected Shimura variety $M=\Gamma\bsh X^+$.
Note that $\Gamma$ serves as the topological fundamental group of $M$. Here  $M$ is not necessarily proper, but if we have chosen $\Gamma$ to be small enough,
then $M$ admits a smooth toroidal compactification $\Mbar$ by joining finitely many divisors,
and the monodromy of $\Gamma\ra\GL_\Qbb(V)$ along each  irreducible component of the boundary $\partial M:=\Mbar- M$ is unipotent.
If the $\Qbb$-PVHS $\Vbb$ consists of a filtered flat connection $(\Vcal,\nabla,Fil^\bullet)$,
then it extends uniquely to a filtered flat connection on $\Mbar$ with logarithmic poles along $\partial M$, which we still denote by $\Vbb=(\Vcal,\nabla,Fil^\bullet)$.

For the $\Qbb$-PVHS $\Vbb$ given above from the representation $(V,\rho)$,
we have the Higgs bundle $(\Ecal,\theta)$ on $M$ by taking the graded quotient of the filtered flat connection
$\Vbb$: $\Ecal=\oplus_r\Ecal^r$ with $\Ecal^r=Fil^r/Fil^{r+1}$ and the Higgs field $\theta:\Ecal^r\ra \Ecal^{r-1}\otimes_{\Ocal_M}\Omega^1_M$ is induced by the filtered flat connection.
Similarly, on the compactification $\Mbar$, the flat connection with logarithmic poles $\nabla:\Vcal\ra\Vcal\otimes_{\Ocal_{\Mbar}}\Omega^1_\Mbar(\partial M)$
together with its filtration $Fil^\bullet$ gives rise to the logarithmic Higgs bundle   on $\Mbar$, which we still write as $(\Ecal=\oplus_r\Ecal^r,\theta)$.

When $(\Gbf,X;X^+)=(\GSp_V,\Hscr_V^\pm;\Hscr_V^+)$ and $\Gamma=\Gamma(\ell)\subset\Sp_V(\Qbb)$ a principal congruence subgroup,
we will be mainly interested in the $\Qbb$-PVHS $\Vbb$ and the Higgs bundle $(\Ecal,\theta)$ on $M=\Acal_V$ given by the standard representation $\rho:\GSp_V\mono\GL_V$.
In this case, write $\pi:\Xcal\ra \Acal_V$ for the universal abelian scheme over $M$ using the moduli interpretation of $\Acal_V$,
we see that the underlying local system of $\Vbb$ is $R^1\pi_*\Qbb_\Xcal$, and the Hodge filtration comes from the canonical exact sequence
$$0\ra \pi_*\Omega^1_{\Xcal/\Acal_V}\ra\Vcal\ra R^1\pi_*\Ocal_\Xcal\ra 0,$$ with $Fil^\inv=\Vcal$ and $Fil^0=\pi_*\Omega^1_{\Xcal/\Acal_V}$.
Hence the associated Higgs bundle   is   $\Ecal=\Ecal^\minuszero\oplus\Ecal^\zerominus$ with
$\Ecal^\minuszero=\Ecal^0=Fil^0=\pi_*\Omega^1_{\Xcal/\Acal_V}$ and $\Ecal^\zerominus=\Ecal^{-1}=Fil^\inv/Fil^0=R^1\pi_*\Ocal_\Xcal$,
while the Higgs field $\theta:\Ecal^{0,-1}\ra\Ecal^{-1,0}\otimes_{\Ocal_{\Acal_V}}\Omega^1_{\Acal_V}$ is equal to the edge morphism of the tautological exact sequence
$$0\ra \pi^*\Omega^1_{\Acal_V}\ra\Omega^1_\Xcal\ra\Omega^1_{\Xcal/\Acal_V}\ra0.$$
The extension of these structures  to a smooth toroidal compactification of $\Acal_V$ is similar, and we do not repeat the details.

From the Simpson correspondence we know that the Higgs bundle $(\Ecal,\theta)$ above on $\Acal_V$ is irreducible,
because it corresponds to the irreducible $\Cbb$-representation of the fundamental group $\Gamma\ra\GL_\Cbb(V_\Cbb)$,
induced by the absolutely irreducible algebraic representation $\Sp_V\ra\GL_V$.
When we consider a Shimura subvariety $M$ defined by $(\Gbf,X;X^+)\mono(\GSp_V,\Hscr_V^\pm;\Hscr_V^+)$, with $M\isom\Gamma_M\bsh X^+$
for $\Gamma_M$ a torsion-free congruence subgroup in $\Gbf^\der(\Qbb)^+$, the $\Cbb$-representation $\Gamma_M\mono\Gbf^\der(\Rbb)^+\ra\GL_\Cbb(V_\Cbb)$ is no longer irreducible in general,
and we can decompose the representation of $\Gamma_M$ to obtain the decomposition of Higgs bundles.
In this case, the ambient representation of $\Gamma_M$ on $V_\Cbb$ corresponds to the Higgs bundle $\Ecal_M$ which is the pull-back of $\Ecal$ along $M\mono\Acal_V$,
and we have an decomposition $\Ecal_M=\Fcal_M\oplus\Ucal_M$, where $\Ucal_M$ is the maximal unitary Higgs subbundle of $\Ecal_M$,
i.e., the Higgs subbundle corresponding to the maximal  subrepresentation of $\Gamma_M\ra\GL_\Cbb(V_\Cbb)$ on which $\Gamma_M$ acts through a compact unitary group.
Usually we assume for simplicity that the level $\ell$ is chosen to be large enough so that the closure
$\Mbar$ of $M$ in a given smooth compactification $\Acalbar_V$ of $\Acal_V$ remains a smooth compactification by joining finitely many divisors,
and the same holds for a fixed Shimura curve $C$ in $M$. In this case the decomposition of Higgs bundles, the maximal unitary Higgs subbundles, etc. extend to $\Mbar$ and $\Cbar$.

We call $\Ecal_M$ resp. $\Ecal_\Mbar$ the canonical Higgs bundle on $M$ resp. on $\Mbar$ given by the embedding $M\mono\Acal_V$ resp. $\Mbar\mono\Acalbar_V$.
Note that the Higgs bundles involved are given by $\Qbb$-PVHS, hence
$$\rank\Ecal^{\minuszero}=\Ecal^\zerominus=\frac{\rank\Ecal}{2},$$
and the same holds for the summands in the decompositions $\Ecal_M=\Fcal_M\oplus\Ucal_M$ and $\Ecal_C=\Fcal_C\oplus\Ucal_C$, as well as their extensions to smooth compactifictions.

\begin{remark}[(Representations of fundamental groups)]\label{representations of fundamental groups}
In this paper, we make use of numerical properties of the decomposition $\Ecal_M=\Fcal_M\oplus\Ucal_M$ as well as its compactified form.
What matters is the rank of $\Ucal_M$, which corresponds to the maximal subrepresentation of $V_\Cbb$ on which $\pi_1(M)$ acts through a compact unitary group.
Note that the representation $\Gamma\mono\GL_\Cbb(V_\Cbb)$ is restricted from $\Gbf^\der(\Rbb)\mono\Sp_V(\Rbb)\mono\GL_\Cbb(V_\Cbb)$,
hence is determined by the algebraic representation $\Gbf^\der\mono\Sp_V$ because $\Gamma$ is Zariski dense in $\Gbf^\der(\Qbb)$.
Similarly, when we study the restriction of the Higgs bundles to a special subvariety $M'=\wp_\Gamma(X'^+)$,
what matters is the representation $\Gamma'\mono\GL_\Cbb(V_\Cbb)$ which factors through $\Gamma'\mono\Gbf'^\der(\Rbb)$ where $M'$ is defined by $(\Gbf',X';X'^+)$:
although $\Gamma'$ is not necessarily equal to $\pi_1(M')$, by Lemma \ref{uniformization of Shimura subvarieties} we know that $\Gamma'$ is a subgroup of finite index in $\pi_1(M')$,
hence to determine the rank of the maximal unitary Higgs subbundle it suffices to study the unitary subrepresentation of $\Gamma'$ on $V_\Cbb$,
determined by the algebraic representation $\Gbf'^\der\ra\Sp_V$.
This is the principle behind the computations concerning symplectic representations in Sections \ref{section Symplectic representations in the unitary case} and \ref{section Symplectic representations in the orthogonal case}.
\end{remark}

\section{Xiao's technique for curves}\label{section Xiao's technique for curves}
In this section we prove Theorem \ref{exclusion of Shimura curves} using the technique of Xiao in \cite{xiao87}.

\subsection{Harder-Narasimhan filtration of a locally free sheaf}\label{subsection Harder-Narasimhan filtration of a locally free sheaf}

In the subsection, we recall the Harder-Narasimhan filtration for a locally free sheaf over an algebraic complete curve.

Let $\Bbar$ be a smooth projective curve over $\Cbb$, and $\Ecal$ a (non-zero) locally free sheaf over $\Bbar$.  The slope of $\cale$ is defined to be the rational number
$$\mu(\Ecal)=\frac{\deg \Ecal}{\rank \Ecal}.$$ $\Ecal$ is said to be stable (resp. semi-stable), if for any coherent subsheaf $0\neq\Ecal'\subsetneq\Ecal$ we have $\mu(\Ecal')<\mu(\Ecal)$ (resp. $\mu(\Ecal')\leq\mu(\Ecal)$);
it is said to be positive (resp. semi-positive), if for any quotient sheaf  $\cale \twoheadrightarrow \calq \neq 0$, one has $\deg \calq>0$ (resp. $\deg \calq\geq0$);
it is said to be poly-stable if it is a direct sum of stable locally free subsheaves of the same slope.
It is clear that any poly-stable locally free sheaf is semi-stable.

 It is well-known (cf. \cite{hardernarasihan745}) that $\cale$ has a unique filtration
 \begin{equation}\label{eqnharder-nara}
 0=\cale_0\subset \cale_1 \subset \cdots \subset \cale_n=\cale,
 \end{equation}
 such that:
 \begin{enumerate}
 \item[(i).] the quotient $\cale_i/\cale_{i-1}$ is a locally free semi-stable sheaf for each $i$;
 \item[(ii).] the slopes are strictly decreasing $\mu(\cale_1/\cale_{0})>\mu(\cale_2/\cale_1)>\cdots>\mu(\cale_n/\cale_{n-1}).$
 \end{enumerate}
 Such a filtration is called the Harder-Narasimhan   filtration of $\cale$. The slope $\mu(\cale_n/\cale_{n-1})$ is called the final slope of $\cale$,  and is denoted by $\mu_f(\cale)$.

 Using the Harder-Narasimhan filtration, one sees easily that $\cale$ is positive (resp. semi-positive, resp. semi-stable),
 if and only if $\mu_f(\cale)>0$ (resp. $\mu_f(\cale)\geq 0$,
 resp. $\mu(\cale)=\mu_f(\cale)$).

 The Harder-Narasimhan filtration is functorial, in the sense that for any finite cover $\varphi:\,\wt B \to \olb$ between two algebraic curves, the pull-back of the Harder-Narasimhan filtration of $\cale$ over $\olb$ to $\wt B$ coincides with the Harder-Narasimhan filtration of $\varphi^*\cale$. Hence the property of positivity (resp. semi-positivity, resp. semi-stability)
 is persevered under any finite cover.

 \subsection{Xiao's technique}\label{subsection Xiao's technique}
 In the subsection, we assume that  $\bar f:\,\ol S \to \ol B$ is a relative curve of genus $g\geq 2$, which is not isotrivial,  and $D$ is a divisor on $\ols$ such that $\cale=f_*\calo_{\ols}(D)$ is a locally free sheaf on $\olb$.

 \begin{definition}[(\cite{xiao87})]\label{defofN(F)}
 Let $\Ecal'$ be a locally free subsheaf of $\cale$.  We define the fixed and moving parts of $\Ecal'$, denoted by $Z(\Ecal')$ and $M(\Ecal')$ respectively, as follows.
 Let $\call$ be a sufficiently ample line bundle on $\olb$ such that the sheaf $\Ecal'\otimes\call$ is generated by its global sections, and $\Lambda\subseteq |\calo_{\ols}(D)\otimes f^*\call|$ be the linear subsystem corresponding to sections in $H^0(\olb,\,\Ecal'\otimes\call)$. Then define $Z(\Ecal')$ to be the fixed part of $\Lambda$, and $M(\Ecal')=D-Z(\Ecal')$.

 Note that the definition above does not depend on the choice of $\call$, and $Z(\Ecal')$ is always effective or zero. We also define $N(\Ecal')=M(\Ecal')-\mu_f(\Ecal')F$, where $F$ is a general fibre of $\bar f$.
 \end{definition}

 An important observation of Xiao is the following, whose proof we refer to \cite[Lemma\,3]{xiao87}.
 \begin{lemma}\label{lemmaN(F)nef}
 For any locally free subsheaf $\Ecal'$ of $\cale$, $N(\Ecal')$ is a nef $\Qbb$-divisor, i.e., $E\cdot N(\Ecal')\geq 0$ for any effective divisor $E$ on $\ols$.
 \end{lemma}

 \subsection{Bound of the unitary part}\label{subsection Bound of the unitary part}
 In the subsection, we consider the unitary part of a relative curve.

 Let $\fbar:\Sbar\ra\Bbar$ be a relative curve of genus $g\geq 2$, which is not isotrivial. Let $\omega_{\ol S/\ol B}$ be the relative canonical sheaf. Then $\fosb$ is a locally free sheaf over $\olb$ of rank $g$. It is well-known (cf. \cite{fujita78}) that $\fosb$ is semi-positive.
Moreover, one has the following decomposition (cf. \cite{fujita78}) $$\bar f_*\omega_{\ol S/\ol B}= \Fcal_\Bbar^\minuszero\oplus\Ucal_\Bbar^\minuszero,$$
where $\Fcal_\Bbar^\minuszero$ is an ample vector bundle over $\ol B$, and $\Ucal_{\olb}^\minuszero$ is a unitary vector bundle,
i.e., a vector bundle corresponds to a unitary representation of the fundamental group $\pi_1(\ol B)$.
 Note that in \cite{fujita78} the superscript is $(1,0)$, and we modify it into $(-1,0)$ following our notations on the Hodge type in the $\Qbb$-PVHS $R^1f_*\Qbb_{S}$,
where $f:\,S \to B$ is the smooth part of $\bar f$.

  The next lemma is essentially due to Xiao (cf. \cite{xiao87}).
 \begin{lemma}\label{lemmaxiaoorignal}
 Conditions as above. Then
 \begin{equation}\label{eqnflatpart1}
 \rank \Ucal_{\olb}^\minuszero \leq \frac{5g+1}{6}.
 \end{equation}
 \end{lemma}
 \begin{proof}
 Let $\cale=\bar f_*\omega_{\ol S/\ol B}$, and consider the Harder-Narasimhan filtration of $\cale$ as in \eqref{eqnharder-nara}. By Definition \ref{defofN(F)},  we have the following decomposition
 \begin{equation}\label{eqnpfflatpart12}
 \omega_{\ol S/\ol B}=Z(\cale_1)+N(\cale_1)+\mu(\cale_1)F.
 \end{equation}
 Hence
 \begin{equation}\label{eqnpfflatpart11}
 12 \deg \fosb \geq \omega_{\ol S/\ol B}^2 \geq \omega_{\ol S/\ol B}\cdot \mu(\cale_1)F=(2g-2)\mu(\cale_1).
 \end{equation}
 By the definition of Harder-Narasimhan filtration, $$\mu(\cale_1)\geq \mu(\Fcal_\Bbar^\minuszero)=\frac{\deg \fosb}{g-\rank \Ucal_{\olb}^\minuszero}.$$ So \eqref{eqnflatpart1} follows immediately from \eqref{eqnpfflatpart11}.
 \end{proof}

The next lemma is a refined upper bound of the rank of the unitary part  if the ample part $\Fcal_\Bbar^\minuszero$ is semi-stable.
 \begin{lemma}[(Bound of the unitary part)]\label{bound of the semi-stable part}
 If $\Fcal_{\olb}^{-1,0}$ is semi-stable, then
 \begin{equation}\label{eqnflatpart2}
 \rank \Ucal_{\olb}^{-1,0} < \frac{4g+2}{5}, \mathrm{\quad and\quad } \rank\Fcal_\Bbar^\minuszero > \frac{g-2}{5}.
 \end{equation}
 \end{lemma}
\begin{proof}
If the general fibre of $f$ is hyperelliptic, then by \cite[Theorem\,4.7]{luzuo14} (see also \cite[Theorem\,A.1]{luzuo13}), we may assume that $\Ucal_{\olb}^{-1,0}$ is trivial after a suitable base change.
Hence according to \cite[Theorem\,1]{xiao92-0} or \cite[Theorem\,1.4]{luzuo13}, we have $\rank \Ucal_{\olb}^{-1,0}=q_\fbar \leq \frac{g+1}{2} < \frac{4g+2}{5}$, where $q_\fbar$ is the relative irregularity of $\fbar$.

Now, we assume that the general fibre of $f$ is non-hyperelliptic.
By assumption, the Harder-Narasimhan filtration of $\cale=\bar f_*\omega_{\ol S/\ol B}$ is of the form $$0=\cale_0\subset \cale_1 \subset \cale_2=\cale.$$ Consider the decomposition as in \eqref{eqnpfflatpart12}.
Let $d_1=N(\cale_1)\cdot F$, and write $Z=Z(\cale_1)$, $N=N(\cale_1)$ and $\mu=\mu(\cale_1)$ for simplicity.
Note that $d_1$ is equal to the degree of a linear system of dimension $g-\rank \Ucal_{\olb}^{-1,0}-1$ on $F$,
and $F$ is non-hyperelliptic.
So by Clifford's theorem, one has
\begin{equation}\label{eqnboundofsemi-stablepart 1}
d_1> 2\left(g-\rank \Ucal_{\olb}^{-1,0}-1\right),\qquad\text{unless $d_1=g-\rank \Ucal_{\olb}^{-1,0}-1=0$}.
\end{equation}
Therefore
 \begin{eqnarray*}
 12 \deg \fosb &\geq&\osb^2\\
 &=& \big(\osb+N\big)\cdot\mu F+\osb\cdot Z+N\cdot \big(N+Z\big)\\
 &\geq& (2g-2+d_1)\mu+\osb\cdot Z\\
 &>& \left(4g-4-2\rank \Ucal_{\olb}^{-1,0}\right)\cdot \mu.
 \end{eqnarray*}
The last inequality follows from \eqref{eqnboundofsemi-stablepart 1} and the fact that $\osb\cdot Z>0$ if $d_1=0$.
Therefore \eqref{eqnflatpart2} follows immediately by noting that $\mu=\frac{\deg \fosb}{g-\rank \Ucal_{\olb}^{-1,0}}$.
\end{proof}

\begin{proof}[Proof of Theorem {\rm \ref{exclusion of Shimura curves}}]
Let $C\subset\Acal_V$ be a   curve, with $V$ a symplectic $\Qbb$-space of dimension $2g$.
Let $\Ecal_\Cbar^\minuszero=\Fcal_\Cbar^\minuszero\oplus\Ucal_\Cbar^\minuszero$
be the decomposition of $(\minuszero)$-part of the Higgs bundle on a smooth compactification $\Cbar$ of $C$,
as the closure of $C$ in some smooth toroidal compactification $\Acalbar_V$ of $\Acal_V$.

If $C$ is contained generically in the Torelli locus $\Tcal_g$, then we have the following commutative diagram
$$\xymatrix{B\ar[r]^{j_B}\ar[d] & C\ar[d]^\cap\\ \Mcal_g\ar[r]^{j^\circ} & \Acal_g}$$
where $B$ is the normalization of the pull-back $(j^\circ)^\inv C$ in $\Mcal_g$.
The morphism $B\ra \Mcal_g$ gives rise to a relative $B$-curve $f:S\ra B$, which is a surface fibred over $B$ by curves of genus $g$.
It admits a compactification into a fibration $\fbar:\Sbar\ra\Bbar$ of semi-stable curves over the smooth compactification $\Bbar$ of $B$,
and the morphism $j_B$ extends naturally to $j_{\ol B}: \Bbar \to \Cbar$.
Let $\omega_{\Sbar/\Bbar}=\omega_{\Sbar}\otimes\fbar^*\omega_\Bbar^{\vee}$ be the relative canonical sheaf of $\bar f$.
Then $\fbar_*\omega_{\Sbar/\Bbar}$ is a locally free sheaf on $\Bbar$ of rank $g$, and it admits a decomposition
$\fbar_*\omega_{\Sbar/\Bbar}=\Fcal^{-1,0}_\Bbar\oplus\Ucal_\Bbar^{-1,0}$.
Note that $\fbar_*\omega_{\Sbar/\Bbar}$ is nothing but the $(-1,0)$-part of the Higgs bundle on $\Bbar$ associated
to the relative semi-stable Jacobian family $jac(\bar f):\,Jac(\Sbar/\Bbar) \to \Bbar$ (cf. \cite[\S\,3]{luzuo14}).
Hence the pull-back by $j_{\Bbar}$ of the decomposition $\Ecal_\Cbar^\minuszero=\Fcal_\Cbar^\minuszero\oplus\Ucal_\Cbar^\minuszero$
coincides with $\fbar_*\omega_{\Sbar/\Bbar}=\Fcal_\Bbar^{-1,0}\oplus\Ucal_\Bbar^\minuszero$,
from which it follows that
$$\rank \Ucal_{\olb}^\minuszero \leq \frac{5g+1}{6},
\quad \text{by Lemma \ref{lemmaxiaoorignal}}.$$
It is a contradiction to our assumption.

When $C$ is a Shimura curve, the summand $\Fcal_{\olc}^{-1,0}$ is poly-stable by Theorem \ref{Arakelov (in)equality},
which implies that $\Fcal_\Bbar^{-1,0}=j_{\Bbar}^*\left(\Fcal_{\olc}^{-1,0}\right)$ is semi-stable.
Hence one obtains the bound $\rank \Ucal_{\olb}^\minuszero < \frac{4g+2}{5}$ according to Lemma \ref{bound of the semi-stable part}, which completes the proof.
\end{proof}
%

\section{Shimura varieties of unitary and orthogonal types}\label{section Shimura varieties of unitary and orthogonal types}

In this section we collect some facts about Shimura varieties of unitary and orthogonal types. The material presented is standard, cf. \cite{kudla orthogonal} \cite{kudla rapoport}.

\subsection{The unitary case}\label{subsection the unitary case}
We start with the construction of unitary groups associated to Hermtian modules.

Let $E/F$ be a quadratic field extension, and $D$ a central simple $E$-algebra, endowed with an involution of second kind $a\mapsto \abar$, i.e., its restriction to the center $E$ gives the non-trivial automorphism of $E$ fixing $F$.
An Hermitian module over $D$ is a right $D$-module $H$ of finite rank endowed with a sesquilinear pairing $(\ ,\ ):H\times H\ra D$ such that $(ua,vb)=\abar(u,v)b$ and $\overline{(u,v)}=(v,u)$ for any $u,v\in H$ and $a,b\in D$.
We always require $(\ ,\ )$ to be non-degenerate, thus $h=\tr_{D/E}\circ(\ ,\ )$ gives a non-degenerate Hermitian form $H\times H\ra E$ over $E/F$.

We thus have the $F$-group $\GU(H/D)$ of unitary similitude, which represents the functor sending an $F$-algebra $R$ to
$$\left\{g\in\Aut_{R\otimes_FE}(R\otimes_FH)~\Bigg|~
\begin{aligned}
&g(ud)=g(u)d,~ \forall\, d\in R\otimes_FD,\mathrm{\ and\ }\\
& g^*g=\nu(g)\in R^\times\subset(R\otimes_FD)^\times,
\end{aligned}\right\},$$
where $g^*$ is the transpose of $g$ \wrt $(\ ,\ )$. The map $g\mapsto g^*g$ gives a surjective $F$-group homomorphism $\nu:\GU(H/D)\ra\Gbb_{\mrm,F}$,
whose kernel is the unitary $F$-group $\Ubf(H/D)$.
The special unitary $F$-group  $\SU(H/D)$ is defined as the derived $F$-subgroup of $\Ubf(H/D)$, and we have the following short exact sequences
$$\begin{aligned}
1&\lra&\Ubf(H/D)~~&\lra&\GU(H/D)&\ot{\nu}\lra& \Gbb_{\mrm,F}\quad&\lra& 1,\\
1&\lra&\SU(H/D)&\lra&\GU(H/D)&\lra&\Res_{E/F}\Gbb_{\mrm,E}&\lra&1.
\end{aligned}$$
Note that
$\Ubf(H/D)$ is an $F$-form of $\GL_N$ and $\SU(H/D)$ is an $F$-form of $\SL_N$ with $N=\dim_EH$.

Assume from now on that $E/F$ is a quadratic CM extension over a totally real field $F$, and write $\tau_1,\cdots,\tau_d$ for the distinct real embeddings.
Then along $\tau_j$ we get the semi-simple $\Rbb$-group $\SU(H_j/D_j)$ isomorphic to $\SU(p_j,q_j)$ because there is no non-trivial division algebra over $\Cbb$. Here $(p_j,q_j)$ is the signature of
$h_j:H_j\times H_j\ra E_j$, and the subscript $j$ indicates the tensor product with $\Rbb$ along $\tau_j$: $H_j=H\otimes_{F,\tau_j}\Rbb \isom \Cbb^N$ and $E_j=E\otimes_{F,\tau_j}\Rbb \isom \Cbb$.
The Hermitian symmetric domain $D(H_j/D_j)$ associated to $\SU(H_j/D_j)$  can be identified with the set of negative definite complex subspaces of dimension
$q_j$ in $H_j$, which is an open subset of the Grassmannian $\Gr_{H_j,q_j}$. We thus get the Shimura datum of unitary type $\left(\Res_{F/\Qbb}\GU(H/D),X\right)$
with  $X=\prod_{j}D(H_j/D_j)$. To embed such a Shimura datum into some Siegel datum,
we prefer to shrink the group to the $\Qbb$-subgroup of rational weights
$$\Gbf=\Gbb_{\mrm,\Qbb}\times_{\Res_{F/\Qbb}\Gbb_{\mrm, F}}\Res_{F/\Qbb}\GU(H/D).$$
Take further $A$ the imaginary part of $h=\tr_{D/E}\circ(\ ,\ )$, i.e., $2\sqrt{\Delta}A(x,y)=2h(x,y)-\tr_{E/F} h(x,y)$
with $\Delta\in F^\times$ some totally negative element such that $E=F(\sqrt{\Delta})$. Then $A$ is an $F$-linear symplectic form on $H$,
and  $\psi=\tr_{F/\Qbb}A$ is a symplectic form over $\Qbb$ on $V=\Res_{D/\Qbb}H$  the $\Qbb$-vector space underlying $H$.
Thus $\Gbf$ preserves $\psi$ up to rational similitude, and $(\Gbf,X)$ is a subdatum of $(\GSp_V,\Hscr_V^\pm)$.

Such Shimura data could contain Shimura subdata of unitary type in the following way.
Assume that $U\subset H$ is a $D$-subspace, such that the restriction of $h$ to $D$ is definite along each $\tau_j$.
Then we have a decomposition $H=U\oplus W$, which is orthogonal \wrt $(\ ,\ )$,
and we have the Levi $F$-subgroup $\GU(U,W/D)$ of $\GU(H/D)$ respecting the decomposition,
whose derived $F$-group is $\SU(U/D)\times\SU(W/D)$. The symmetric domain it gives along $\tau_j$ is $D(W_j/D_j)$,
because the definite Hermitian space $U_j$ gives a compact $\Rbb$-group $\SU(U_j/D_j)$ which does not contribute to the symmetric domain.
Since in the definition of Shimura data, the derived $\Qbb$-group of the reductive $\Qbb$-group is assumed to have no compact $\Qbb$-factors,
we should remove the part $\SU(U/D)$ when constructing the Shimura subdatum.
Therefore the Shimura subdatum in $(\Gbf,X)$ is of the form $(\Gbf^W,X^W)$ with $(\Gbf^W)^\der\isom\Res_{F/\Qbb}\SU(W/D)$ and $X^W\isom\prod_jD(W_j/D_j)$.

We will only work with a special class of Shimura varieties of untary type, which contains Shimura curves using the procedure above involving a  definite subspace, and induces  a simple Hermitian symmetric domain:

\begin{definition}[(Shimura data of $\SU(n,1)$-type)]\label{Shimura data of SU(n,1)-type} 

A Shimura datum of $\SU(n,1)$-type is a Shimura datum of unitary type given by an Hermitian space $h:H\times H\ra E$ over some quadratic CM extension $E/F$ of dimension $n+1$ (with $D=E$),
such that $h$ is of signature $(n,1)$ along $\tau=\tau_1$, and definite along $\tau_2,\cdots,\tau_d$. We thus get the Shimura datum $(\Gbf^H,X^H)$
where $\Gbf^H$ is the extension of $\Res_{F/\Qbb}\Ubf(H/E)$ by $\Gbb_{\mrm,\Qbb}$,
and $X^H=D(H_\tau)$ is a complex ball identified with the open subset of negative definite complex lines in the complex projective space $\Pbb(H_\tau)$.
\end{definition}

Here only the trivial division $E$-algebra $D=E$ is involved, and we will write $\SU(H)$ in place of $\SU(H/E)$ in the sequel.
Using the Gram-Schmidt arguments, under some $E$-basis $e_1,\cdots,e_{n+1}$ we have
$h\left(\sum x_je_j,\sum y_je_j\right)=\sum a_j\xbar_jy_j$ for some $a_1,\cdots,a_{n+1}\in F^\times$,
and along $\tau$ only one of the $a_j$'s is negative. We may thus construct definite subspaces $U$ of dimension $n-1$ in $H$
which is positive along $\tau$, and an orthogonal decomposition $H=U\oplus W$ which gives rise to the Shimura subdatum of unitary type
$(\Gbf^W,X^W)$ in $(\Gbf^H,X^H)$, with $X^H$ of dimension 1, namely $X^H=\Hscr_1^\pm$ is the Poincar\'e double half space.

\begin{remark}\label{exclusion of general division algebra}
If $D$ is a  non-split central division $E$-algebra, then one cannot construct a Shimura curve using the orthogonal complement of a positive definite $D$-subspace.
In fact, if there is such one with Hermitian structure $(\ ,\ ):W\times W\ra D$, then we have $h=\tr_{D/E}\circ(\ ,\ ):W\times W\ra E$ of signature $(1,1)$ along one real place of $F$.
This force $W$ to be of dimension 2 over $E$, but a non-split central division $E$-algebra does not admit any non-trivial $E$-representation of dimension 2:
$D$ is of dimension at least 4 over $E$, and the $E$-dimension of a right $D$-module has to be a multiple of $\dim_E D$.
\end{remark}

Just as  in the general case, the datum $(\Gbf^H,X^H)$ can be realized as a subdatum of $(\GSp_V,\Hscr_V^\pm)$
where $V=\Res_{E/\Qbb}H$ is the $\Qbb$-vector space underlying $H$ endowed with the symplectic
$\Qbb$-form $\psi=\tr_{F/\Qbb}A$, $A$ being the imaginary part of $h$.
We thus encounter the following symplectic representation over $\Qbb$: $\Gbf^\der=\Res_{F/\Qbb}\SU(H)\ra \Sp_V$,
which we call the standard symplectic representation over $\Qbb$ associated to the Hermitian space $H$.

\subsection{The orthogonal case}\label{the orthogonal case}We first recall a few facts about quadratic spaces and their associated Clifford algebras, and we work over a field of characteristic zero for simplicity.

Let $F$ be a field of characteristic zero and let $(H,Q)$ be a quadratic space over $F$, i.e., $H$ is an $N$-dimensional $F$-vector space, and $Q:H\times H\ra F$ is symmetric bilinear form, which is assumed to be non-degenerate for simplicity. We then have the special orthogonal $F$-group $\SO(H)$.

To obtain a simply-connected $F$-group isogeneous to $\SO(H)$, we need the Clifford algebra $C(H)$ of $(H,Q)$.
It is the quotient $F$-algebra $C(H)=\bigoplus_{r\geq 0}H^{\otimes r}/I_Q$, where $I_Q$ is the homogeneous bilateral ideal generated by elements of the form $v\otimes v-Q(v,v)$ with $v$ running through $H$.
$C(H)$ is an $F$-algebra of dimension $2^N$. The ideal $I_Q$ is generated by elements of even degree, and this induces on $C(H)$ a $\Zbb/2$-grading $C(H)=C^+(H)\oplus C^-(H)$, with $C^+(H)$ (resp.
$C^-(H)$) the $F$-subspace generated by the images of elements of even degree (resp. odd degree), both of which are of dimension $2^{N-1}$ over $F$.
Note that when $(H,Q)\isom(H',Q')\oplus(H'',Q'')$ is an orthogonal decomposition of quadratic spaces,
we naturally have $C(H)\isom C(H')\otimes C(H'')$ as $\Zbb/2$-graded tensor products of $\Zbb/2$-graded algebras.

The Clifford algebra carries an involution $a\mapsto a^*$, induced by the involution on $\bigoplus_{r\in\Nbb}H^{\otimes r}$ generated by $v_1\otimes\cdots\otimes v_r\mapsto(-v_r)\otimes\cdots\otimes(-v_1)$.
We thus get the $F$-group of spin similitude $\GSpin(H)$, which sends an $F$-algebra $R$ to $$\left\{g\in\Aut_R(R\otimes_FC^+(H):~g(H\otimes_FR)g^\inv\subset H\otimes_FR,\nu(g)=g*g\in R^\times \right\},$$
where we use the canonical inclusion $H\mono C(H)$ for the expression $g(H\otimes_FR)g^\inv$.
Its derived $F$-group is the spin group $\Spin(H)$, fitting into the exact sequence $$1\lra\Spin(H)\lra\GSpin(H)\ot{\nu}\lra\Gbb_{\mrm, F}\lra 1,$$ and it is also a central extension of $\SO(V)$ by $\mu_2=\{\pm\}$.

Note that over $\Rbb$ a quadratic space is determined by its signature $(n,m)$ i.e., the dimensions of the positive/negative parts, and in this case we also write $\Spin(n,m)$, $\GSpin(n,m)$ as we do with $\SO(n,m)$ etc. When $m=2$ and $n\geq 1$, the symmetric domain $D=D(H)=D(n,2)$ defined by $\Spin(H)=\Spin(n,2)$ is Hermitian: it is identified with the set of the isotropic negative definite complex lines in the complex projective space $\Pbb(H_\Cbb)$, which is an open subset of the $n$-dimensional quadric in $\Pbb(H_\Cbb)$ defined by $Q$.

\begin{definition}[(Shimura data of $\SO(n,2)$-type)]\label{Shimura data of SO(n,2)-type}
By a Shimura datum of $\SO(n,2)$-type we mean a Shimura datum of the form $(\Gbf,X)$ with $\Gbf^\der=\Res_{F/\Qbb}\Spin(H)$   where $(H,Q)$ is a quadratic space of dimension $n+2$ ($n\geq 1$) over a totally real field $F$ of degree $d$ such that \begin{itemize}
\item $(H,Q)$ is of signature $(n,2)$ along one embedding $\tau=\tau_1:F\ra\Rbb$, and definite along the other ones $\tau_2,\cdots,\tau_d$;

\item $X=D(H_\tau)$ is the Hermitian symmetric domain associated to the real quadratic space $H_\tau$ of signature $(n,2)$.
\end{itemize}
This makes sense because $\Gbf^\der(\Rbb)\isom\Spin(n,2)\times(\Spin(n+2,0))^{d-1}$ and the compact part $(\Spin(n+2,0))^{d-1}$ does not contribute to the Hermitian symmetric domain defined by $\Gbf^\der(\Rbb)$.
\end{definition}

Similar to the unitary case we can construct Shimura subdata in $(\Gbf,X)$ of $\SO(n,2)$-type in each codimension. In fact, let $U\subset H$ be an $F$-subspace of dimension $m$ ($0\leq m\leq n-1$), positive definite along $\tau$ under $Q$, and let $H=U\oplus W$ be the orthogonal decomposition \wrt $Q$. Then $(W,Q)$ is a quadratic space over $F$ of signature $(n-m,2)$ along $\tau$, and definite along the other embeddings. This gives us the Shimura subdatum $(\Gbf',D(W_\tau))\subset(\Gbf,D(H_\tau))$, where $D(W_\tau)\subset D(H_\tau)$ is identified with the subset of isotropic negative definite complex lines in $\Pbb(H_\tau\otimes_\Rbb\Cbb)$ orthogonal to $U_\tau\otimes_\Rbb\Cbb$, i.e., contained in $W_\tau\otimes_\Rbb\Cbb$. The construction of $\Gbf'$ is similar to the unitary case we have explained before: it is obtained by removing the normal $\Qbb$-subgroup $\Res_{F/\Qbb}\Spin(U)$ from the Levi $\Qbb$-subgroup corresponding $U\oplus W$, and $\Gbf'^\der\isom\Res_{F/\Qbb}\Spin(W)$. Take $m=n-1$ we get Shimura subdata defining Shimura curves.

\begin{remarks}
(1) Similar to the unitary case, one could start with a central division $F$-algebra $D$ endowed with an involution $a\mapsto \abar$ of first kind, i.e., whose restriction to $F$ is the identity, and consider a non-degenerate Hermitian module $Q:H\times H\ra D$. In this way the composition $q=\tr_{D/F}Q$ is a quadratic form on $H$, and one gets a unitary $F$-group $\SU(H/D)$ which is an $F$-form of a special orthogonal $F$-group. In this setting we can also consider Shimura data with $F$ a totally real field of degree $d$, under the assumption that along each embedding $\tau_j:F\ra \Rbb$ the base change $q_j:H_j\times H_j\ra F_j$ is either definite or of signature $(n,2)$ or $(2,n)$, and the symmetric domain involved is isomorphic a finite product of copies of the one associated to $\SO(n,2)$.

Inspired by Hain's results \cite{hain99} we only consider the case of one single Hermitian symmetric domain of $\SO(n,2)$-type,
hence only one of the real embeddings is assumed to produce a non-compact Lie group $\Spin(n,2)$,
so that we have Shimura subvarieties in each codimension. Moreover, what matters for us is the existence of Shimura curves.
In the setting involving non-split central division $F$-algebra $D$ which is of dimension at least 4,
we cannot repeat the procedure in the split case of cutting out Shimura curves using definite $D$-subspaces in $H$:
by arguments parallel to Remark \ref{exclusion of general division algebra},
the remaining rank 1 $D$-space gives at most a quadratic space of signature $(2,2)$,
and it leads to a Shimura surface which might fail to contain a Shimura curve,
and our theorem in the curve case no longer applies.

(2) In the Definition \ref{Shimura data of SO(n,2)-type} we have used the spin groups instead of the orthogonal groups, because in this paper we only care about Shimura subdata of $(\GSp_V,\Hscr_V^\pm)$ for some symplectic space $V$ over $\Qbb$. The $\Qbb$-group $\Res_{F/\Qbb}\SO(H)$ is not simply-connected, and in order to have faithful symplectic representations we are forced to use its simply-connected covering $\Res_{F/\Qbb}\Spin(H)$.
\end{remarks}

\section{Symplectic representations in the unitary case}\label{section Symplectic representations in the unitary case}
In this section we first recall some facts about symplectic representations and their primary decomposition, following \cite{satake embedding} and \cite{satake classification},
and then we focus on the case of $\SU(n,1)$-type and the restriction to Shimura curves.
Using the restrictions of symplectic representations to Shimura curves together with Theorem \ref{exclusion of Shimura curves},
we prove our main result for the unitary case in Section \ref{proof of the main result in the unitary case}.

\subsection{Real symplectic representations in general}\label{subsection Real symplectic representations}
Let $G$ be a connected non-compact semi-simple Lie group. It is said to be of Hermitian type if the non-compact symmetric space it defines is an Hermitian symmetric domain. For example, if $(V,\psi)$ is a symplectic space over $\Rbb$, then the symplectic group $\Sp(V)$ is of Hermitian type, and the Hermitian symmetric domain is the (connected) Siegel upper half space $\Hscr_V^+\isom\Sp(V)/\Urm(V)$ where $\Urm(V)$ is the unitary group for some positive definite Hermitian form $h$ on $V$, such that the imaginary part of $h$ is equal to $\psi$.

We are interested in the embeddings of Hermitian symmetric domains of the form $i:X^+\mono\Hscr_V^+$,
where $X^+$ is given by some Lie group of Hermitian type $G$, i.e., $X^+=G/K$ for some maximal compact subgroup $K$ of $G$, such that \begin{itemize}
\item $i$ is holomorphic, and $i(X^+)$ is totally geodesic in $\Hscr_V^+$;
\item $i$ is equivariant form some inclusion of Lie groups $\rho:G\mono\Sp(V)$;
\end{itemize}Actually we impose more constraints: let $H_0\in\Lie K$ resp. $H_0'\in\Urm(V)$ be the unique element defining the complex structure on $X^+$ resp. on $\Hscr_V^+$, and we require further\begin{enumerate}
\item[(H1)] $\rho\circ\ad(H_0)=\ad(H'_0)\circ\rho$;\\ or the even stronger condition \item[(H2)] $\rho(H_0)=H'_0$.
\end{enumerate} When we have a second embedding $(i',\rho')$ subject to these conditions, we regard $(i,\rho)$ and $(i',\rho')$ as (k)-equivalent  if for some $k'\in \Urm(V)$ we have $\rho'=\ad(k')\circ\rho$ on the Lie algebras.

The representations completely determine the embeddings, and when the target is $(\Sp(V),\Hscr_V^+)$ for some real symplectic space $V$, they are referred to as real symplectic representations.
In \cite{satake embedding} Satake has obtained the classification of these embeddings and representations up to (k)-equivalence under the stronger condition (H2).
For every such embedding $(i,\rho)$, the representation $\rho:G\ra\Sp(V)$ splits into a sum of irreducible symplectic subrepresentations $\rho=\bigoplus\limits_{j=0,\cdots,t}\rho_j$,
with $\rho_0:G\ra\Sp(V_0)$ a symplectic representation into some symplectic subspace $(V_0,\psi_0)$ of $V$,
and $\rho_j:G\ra\SU(p_j,q_j)$ ($j=1,\cdots,t$) a representation into some special unitary group of some Hermitian space $(V_j,h_j)$ of signature $(p_j,q_j)$,
such that the sum of $\psi_0$ and the imaginary parts $\psi_j$ of $h_j$ gives the symplectic form $\psi$ on $V\isom\oplus_jV_j$ (cf. \cite[\S\,2, Proposition 3]{satake embedding}).
Note that in these components we do allow trivial representations.

In the case of $\SU(n,1)$-type, all the non-trivial real symplectic representations are essentially copies of the wedge representations  cf.\cite[\S\,3.2]{satake embedding}:
\begin{theorem}[(Irreducible real symplectic representations of $\SU(n,1)$)]\label{irreducible real symplectic representations of SU(n,1)}
Let $(H,h)$ be an Hermitian space over $\Cbb/\Rbb$ of signature $(n,1)$, and let $\SU(H)$ be its special unitary group, with $D=D(H)$ the Hermitian symmetric domain defined by $\SU(H)$,
Then the irreducible real symplectic representations are exactly the natural representations $\Lambda_m$ ($m=1,\cdots,n$) of $\SU(V)$
on the complex vector space $\wedge^m_\Cbb V$ {\rm (}viewed as a real vector space of dimension $2\binom{n+1}{m}${\rm )},
where the symplectic forms are the imaginary parts of canonically defined Hermitian forms of signature $\left(\binom{n}{m},\binom{n}{m-1}\right)$ preserved by the action of $\SU(H)$.
\end{theorem}

In particular, when $n=1$, the identity representation is the only non-trivial irreducible real symplectic representation of $\SU(1,1)$.

\subsection{Primary symplectic representations}\label{subsection Primary symplectic representations}
To define Shimura subvarieties in $\Acal_V$, we need connected Shimura subdata $(\Gbf,X;X^+)\mono(\GSp_V,\Hscr_V^\pm;\Hscr_V^+)$. The embedding $X^+\mono\Hscr_V^+$ is given by the real symplectic representation $\Gbf^\der(\Rbb)\mono\Sp_V(\Rbb)$, which is the evaluation over $\Rbb$ of the $\Qbb$-group homomorphism $\Gbf^\der\mono\Sp_V$. We thus need further results in \cite{satake classification} on the classification of symplectic representations over $\Qbb$.

We recall a few general notions on representations of reductive algebraic groups. Let $k$ be a field of characteristic zero, and let $\Gbf$ be a reductive $k$-group. By representations of $\Gbf$ we mean finite-dimensional algebraic representations of $\Gbf$, i.e., $k$-group homomorphisms of the form $\rho:\Gbf\ra\GL_V$ for some finite-dimensional $k$-vector space $V$. We thus have the $k$-linear abelian  category $\Rep(\Gbf/k)$ of representations of $\Gbf$, with morphisms being $\Gbf$-equivariant $k$-linear homomorphisms. Since $\Gbf$ is assumed to be reductive, the category $\Rep(\Gbf/k)$ is semi-simple, i.e., every short exact sequence splits and therefore every representation is the direct sum of irreducible representations. Note that the endomorphism ring $\End_\Gbf(\rho)$ of any $\rho\in\Rep(\Gbf/k)$ is a finite-dimensional semi-simple $k$-algebra, and $\rho$ is irreducible \ifof $\End_\Gbf(\rho)$ is a finite-dimensional division $k$-algebra.

For $K/k$ a field extension, we have the base change $\Rep(\Gbf/k)\ra\Rep(\Gbf_K/K)$, sending $(V,\rho)$ to $(V_K,\rho_K)$. A $K$-irreducible representation of $\Gbf$ is $(V,\rho)\in\Rep(\Gbf/k)$ such that $(V_K,\rho_K)$ is irreducible in $\Rep(\Gbf_K/K)$. When $K=\kbar$ is an algebraic closure of $k$, we get the notion of absolutely irreducible representations of $\Gbf$, and it is clear that they are exactly the $K$-irreducible representations whenever $K$ is a field extension of $k$ containing an algebraic closure of $k$.

For a fixed absolutely irreducible representation $(\Lambda,\lambda)$ in $\Rep(\Gbf_\kbar/\kbar)$,
a representation $(V,\rho)\in\Rep(\Gbf/k)$ is said to be primary of type $(\Lambda,\lambda)$
if there is an isomorphism $(V_\kbar,\rho_\kbar)\isom(\Lambda,\lambda)^{\oplus N}$ in $\Rep(\Gbf_\kbar/\kbar)$ for some integer $N$.
It is thus obvious that irreducible representations in $\Rep(\Gbf/k)$ are primary,
i.e., their base change to $\kbar$ are direct product of copies of a single absolutely irreducible representation over $\kbar$.

Now let $(\Gbf,X)\subset(\GSp_V,\Hscr_V^\pm)$ be a Shimura subdatum of $\SU(n,1)$-type with $\Gbf^\der=\Res_{F/\Qbb}\SU(H)$,
and we assume for simplicity that its action on $V$ admits no trivial subrepresentations. By the classification in \cite{satake classification} (especially Section 7),
we know that the representation $\rho:\Gbf^\der\ra\Sp_V$ is symplectic, and it is a scalar restriction, i.e.,
there exists a symplectic representation $\lambda:\, \SU(H) \to \Sp_L$ over $F$ with $L$ some symplectic space over $F$ and $\rho=\Res_{F/\Qbb}\lambda$.
In particular, $V=\Res_{F/\Qbb}L$ is the $\Qbb$-vector space underlying $L$, and $\dim_\Qbb V=d\dim_FL$ with $d=[F:\Qbb]$.
Since $F\otimes_\Qbb\Rbb=\prod\limits_{j=1,\cdots,d} \Rbb$, the base change $\rho_\Rbb$ decomposes as $\prod\limits_{j=1,\cdots,d}\lambda_j$
with $\lambda_j:\SU(H_j)\ra\Sp_{L_j}$ is the base change of $\lambda$ along $\tau_j:F\mono\Rbb$, and $L_j$ being the base change of $L$ by $\tau_j$.
Hence on $V\otimes_\Qbb\Rbb\isom\oplus_jL_j$, $\Gbf^\der_\Rbb\isom\prod_j\SU(H_j)$ acts on $L_j$ through $\SU(H_j)$.

\begin{definition}[(Primary type)]\label{primary type} Let $\rho=\Res_{F/\Qbb}\lambda$ be as above, and when $j=1$ we use the subscript $\tau$ in place of $1$. The representation $\rho$ is said to be $\tau$-primary of type $\Lambda_m$ for some fixed $m\in\{1,\cdots,n\}$, if the component $\lambda_\tau$ is isomorphic to a direct sum of copies of $\Lambda_m$: $\lambda_\tau\isom\Lambda_m^{\oplus N}$ for some integer $N>0$.

In particular, we can view $(L,\lambda)$ as an $F$-form of the $N$-fold direct sum of the natural representation of $\SU(H)$ on $\wedge^m_EH$.
\end{definition}

It follows from the classification of Satake that every symplectic representation of $\Res_{F/\Qbb}\SU(H)$ is the direct sum of a trivial representation with finitely many $\tau$-primary representations of various types $\Lambda_m$ ($m\in\{1,\cdots,n\}$).

We are thus lead to the following (cf. \cite[\S\,8.3]{satake classification}):

\begin{proposition}[(Primary representation of $\SU(n,1)$-type)]\label{primary decomposition over a Shimura subvariety of SU(n,1)-type}
Let $(\Gbf,X;X^+)\subset(\GSp_V,\Hscr_V^\pm;\Hscr_V)$ be a connected Shimura subdatum of $\SU(n,1)$-type with $V$ a symplectic space over $\Qbb$ of dimension $2g$,
defining a Shimura subvariety $M\isom\Gamma_M\bsh X^+$, with $\Gamma_M$ a torsion-free congruence subgroup of $\Gbf^\der(\Qbb)^+$.
Let $\Ecal_M$ be the $\Qbb$-PVHS on $M$ induced by the universal abelian scheme over the Siegel moduli scheme defined by $(\GSp_V,\Hscr_V^\pm)$, and $\Ecal_M=\Fcal_M\oplus\Ucal_M$ a decomposition into Higgs bundles with $\Ucal_M$ the maximal unitary part.
Assume that $\Gbf^\der\isom\Res_{F/\Qbb}\SU(H)$ is given by some Hermitian space over a CM extension $E/F$ of dimension $n+1$ with $F$ some totally real number field of degree $d$
such that the symplectic representation $\rho:\Gbf^\der\ra\Sp_V$ is $\tau$-primary of type $\Lambda_m$ for some $m\in\{1,\cdots,n\}$, without trivial subrepresentations.
Then
$$\rank \Ucal_M^\minuszero =\frac{g(d-1)}{d}, \qquad \rank \Fcal_M^\minuszero =\frac{g}{d}.$$
\end{proposition}
\begin{proof}
We need to decompose the representation of fundamental group $\Gamma_M\ra\Gbf^\der(\Cbb)\ra\Sp_V(\Cbb)$, which is extended from the real representation $\Gamma_M\ra\Gbf^\der(\Rbb)\ra\Sp_V(\Rbb)$. Since $\Gamma_M\subset\Gbf^\der(\Rbb)$ is Zariski dense, it suffices to decompose the real symplectic representation of algebraic groups $\rho_\Rbb:\Gbf^\der_\Rbb=\prod_j\SU(H_j)\ra\Sp_{V,\Rbb}$.

As we have seen, we can find an $F$-model $L$ such that $\Gbf^\der\ra\Sp_V$ is restricted from $\lambda:\SU(H)\ra\Sp_L$ and the real symplectic representation is $\rho_\Rbb=\oplus_j\lambda_j$.
For $j=2,\cdots,d$, the $\Rbb$-group $\SU(H_j)$ is compact, and the action of $\Gamma_M$ on $\bigoplus\limits_{j=2,\cdots,d}L_j$ through a product of compact unitary groups on the $L_j$'s,
which can be easily made into a subgroup of some compact unitary group on $\bigoplus\limits_{j\neq 1}L_j$. On the other hand, for $j=1$ and $\tau=\tau_1$,
the $\Rbb$-group $\SU(H_\tau)$ is simple and non-compact and $\lambda_\tau=\lambda\otimes_{F,\tau}\Rbb$ is a direct sum of copies of $\Lambda_m$, hence $\lambda_\tau$ has no contribution to the unitary part.

When we take the further base change $\Rbb\mono\Cbb$, the action of $\Gamma_M$ on $V\otimes_\Qbb\Cbb$ factors through $V\otimes_\Qbb\Rbb$,
hence the part $\bigoplus\limits_{j=2,\cdots,d}L_j\otimes_\Rbb\Cbb$ remains unitary,
and the part $L_\tau\otimes_\Rbb\Cbb$ splits into a conjugate pair of two copies of $L_\tau$ viewed as $\Cbb$-vector spaces,
with no contribution to the unitary part.

We thus conclude that the unitary part $\Ucal_M$ corresponds to the action of $\Gamma_M$ on  $\bigoplus\limits_{j=2,\cdots,d}L_j\otimes_\Rbb\Cbb$, whose complex dimension is $(d-1)\dim_FL$.
Since $V=\Res_{F/\Qbb}L$ is of dimension $2g=d\dim_FL$, we conclude that $$\rank\Ucal_M^\minuszero=\frac{(d-1)\dim_FL}{2}=\frac{g(d-1)}{d},
\text{\quad and \quad} \rank \Fcal_M^\minuszero =g-\rank\Ucal_M^\minuszero=\frac{g}{d}.$$
The proof is complete.
\end{proof}

A formula of general symplectic representations is also obtained:
\begin{corollary}[(General representation of $\SU(n,1)$-type)]\label{formula of SU(n,1)-type}
Let $(\Gbf,X;X^+)\subset(\GSp_V,\Hscr_V^\pm;\Hscr_V^+)$ be a Shimura subdatum of $\SU(n,1)$-type,
defining the Shimura subvariety $M$. Assume that $\Gbf^\der=\Res_{F/\Qbb}\SU(H)$ is given by an Hermitian space $H$ over a CM extension $E/F$ subject to the constraints of signature,
with $\tau:F\mono\Rbb$ the only embedding contributing to the symmetric domain $X$. Assume that $$V=V_0\oplus\Res_{F/\Qbb}\lambda_1\oplus\cdots\oplus\Res_{F/\Qbb}\lambda_n$$
where $V_0$ is a trivial representation of $\Gbf^\der$, and $\lambda_m:\SU(H)\ra\Sp_{L_m}$ is an $F$-linear symplectic representation $\tau$-primary of type $\Lambda_m$, $m=1,\cdots,n$.
Let $\Ecal_M=\Fcal_M\oplus\Ucal_M$ be the decomposition of Higgs bundles of the representation of $\pi_1(M)$ induced from $\rho_\Cbb$, with $\Ucal_M$ the maximal unitary Higgs subbundle.
Then
$$\rank\Ucal_M^{-1,0}=\frac{\dim_\Qbb V_0+(d-1)\cdot\sum\limits_{m=1}^n\dim_FL_m}{2},
\qquad \rank\Fcal_M^{-1,0}=\frac{\sum\limits_{m=1}^n\dim_FL_m}{2}.$$
\end{corollary}

\begin{proof}
It suffices  to point out that the trivial subrepresentation $V_0$ of $\Gbf^\der$ only contributes to a trivial representation of $\pi_1(M)$. 
\end{proof}

\subsection{Restriction to a Shimura curve}\label{subsection Restriction to a Shimura curve}

To compute the ranks in the decomposition of the Higgs bundle $\Ecal_C$ on a Shimura curve $C$ embedded in a Shimura subvariety of $\SU(n,1)$-type, we again reduce to the decomposition over $\Rbb$. We first consider the following real irreducible case:

\begin{lemma}[(Real irreducible case of type $\Lambda_m$)]\label{the real irreducible case}
  Let $H$ be an Hermitian space over $\Cbb/\Rbb$ of signature $(n,1)$, and $H=U\oplus W$ an orthogonal decomposition involving a positive definite subspace $U$ of dimension $n-1$ and $W$ a subspace of signature $(1,1)$, giving rise to the inclusion $\SU(W)\mono\SU(H)$. Let $\Lambda_m$ be the symplectic representation of $\SU(H)$ on $\wedge^m_\Cbb H$. Then the restriction of $\Lambda_m$ to $\SU(W)$ decomposes as $T\oplus\Std^{\oplus r}$, where $T$ is a trivial symplectic representation, and $\Std$ is the identity symplectic representation of $\SU(1,1)$, with $r=\binom{n-1}{m-1}$.

\end{lemma}

\begin{proof}
We first note that when $m=1$, the action of $\SU(1,1)$ on $\Lambda_1=H$ factors through $\SU(U)\times\SU(W)\mono\SU(H)$, hence $\Lambda_1=H=U\oplus W$ is already of the claimed form.

For general $m$, it suffices to notice that $$\wedge^m(U\oplus W)=\bigoplus_{j=0,1,2}\wedge^{m-j}U\otimes\wedge^jW$$ because $W$ is of dimension 2 over $\Cbb$, hence $\wedge^{m-1}U\otimes W\isom W^{\oplus r}$ is the only non-trivial part in $\Lambda_m$ for the restriction to $\SU(W)$, with $r=\binom{n-1}{m-1}$.
\end{proof}

Adapt this lemma to the general case of $\tau$-primary symplectic representations of type $\Lambda_m$, we get:

\begin{proposition}[(Decomposition over a curve)]\label{decomposition over a curve} Let $C\subset M\subset\Acal_V$ be an inclusion chain of Shimura subvarieties given by the chain of Shimura subdata $(\Gbf^W,X^W)\subset(\Gbf^H,X^H)\subset(\GSp_V,\Hscr_V)$, where \begin{itemize}
\item $(\Gbf^H,X^H)$ is a Shimura datum of $\SU(n,1)$-type, given by some Hermitian space $H$ over a CM extension $E/F$, and $\Gbf^\der=\Res_{F/\Qbb}\SU(H)$

\item $(\Gbf^W,X^W)$ is the Shimura subdatum of $\SU(1,1)$-type in $(\Gbf^H,X^H)$ cut out by a positive definite subspace $U$ of dimension $n-1$;

\item $V=\Res_{F/\Qbb}L$ for some symplectic space $L$ over $F$, and the representation $\Gbf^\der\ra\Sp_V$ is restricted from some $\tau$-primary symplectic representation $\SU(H)\ra\Sp_L$ of type $\Lambda_m$.
\end{itemize}
Let $\Ecal_C$ be the Higgs bundle induced by the universal abelian scheme $\Xcal\ra\Acal_V$,
which decomposes into $\Ecal_C=\Fcal_C\oplus\Ucal_C$ with $\Ucal_C$ the maximal unitary Higgs subbundle. Then
$$\rank\Ucal_C^\minuszero=\dfrac{\dim_\Qbb V\cdot \left(\frac{d}{2}\cdot \binom{n+1}{m}-\binom{n-1}{m-1}\right)}{d\cdot \binom{n+1}{m}}, \qquad
\rank\Fcal_C^{\minuszero}=\dfrac{\dim_\Qbb V\cdot \binom{n-1}{m-1}}{d\cdot \binom{n+1}{m}},
\qquad $$
with $d=[F:\Qbb]$.
\end{proposition}

\begin{proof}
Similar to  the proof of Proposition \ref{primary decomposition over a Shimura subvariety of SU(n,1)-type},
it suffices to study how the algebraic representation $\rho_\Rbb:\Res_{F/\Qbb}\SU(H)_\Rbb\ra\Sp_{V,\Rbb}$ decomposes when restricted to $\Res_{F/\Qbb}\SU(W)_\Rbb$.
Since $\rho=\Res_{F/\Qbb}\lambda$ and $\rho_\Rbb=\bigoplus\limits_{j=1,\cdots,d}\lambda_j$ is a direct sum, with $\Gbf^\der_\Rbb$ acts on $\bigoplus\limits_{j=2,\cdots,d}L_j$ through a compact quotient,
we deduce that the pull-back of the unitary part $\Ucal_M$ to $C$ remains unitary.

For the remaining part $\tau=\tau_1$, the representation $\lambda_\tau=\lambda_1:\SU(H_\tau)\ra\Sp_{L_\tau}$ decomposes as $L_\tau\isom\Lambda_m^{\oplus N}$ for some integer $N\geq 1$.
Restrict it to $\SU(W_\tau)$, we are reduced to the case of Lemma \ref{the real irreducible case}:
each $\Lambda_m$ is the sum of a non-trivial subrepresentation $\wedge^{m-1}U_\tau\otimes W_\tau\isom\Std^{\oplus r}$ and a trivial representation, with $r=\binom{n-1}{m-1}$.
Hence the rank of $\Fcal_C$ is $2N\binom{n-1}{m-1}$. Since $L_\tau\isom\Lambda_m^{\oplus N}$, we have $\dim_FL=N\binom{n+1}{m}$,
which gives $$\rank\Fcal_C^\minuszero=\frac{1}{2}\rank\Fcal_C=\frac{\dim_FL\cdot\binom{n-1}{m-1}}{\binom{n+1}{m}}.$$
The proof is complete by noting that $\rank\Ucal_\Cbar^\minuszero=g-\rank\Fcal_\Cbar^\minuszero$ and $\dim_\Qbb V=d\cdot\dim_FL$.
\end{proof}

Similarly we have a general formula:

\begin{corollary}[(Formula of $\SU(1,1)$-type)]\label{formula of SU(1,1)-type}
Let $M\subset\Acal_V$ be the same as in Corollary \ref{formula of SU(n,1)-type},
and $C\subset M$ the Shimura curve cut out by $U$ with $U$ positive definite of dimension $n-1$ along $\tau$.
Write $\Ecal_C=\Fcal_C\oplus\Ucal_C$ be the Higgs bundle decomposition on $C$ with $\Ucal_C$ the maximal unitary Higgs subbundle.
Then $$\begin{aligned}
\rank\Ucal_C^\minuszero&=\frac{1}{2}\dim_\Qbb V_0+\sum_{m=1}^n\frac{\dim_FL_m\cdot\left(\frac d2\cdot \binom{n+1}{m} -\binom{n-1}{m-1}\right)}{\binom{n+1}{m}},\\
\rank\Fcal_C^\minuszero&=\sum_{m=1}^n\dfrac{\dim_FL_m\cdot\binom{n-1}{m-1}}{\binom{n+1}{m}}.
\end{aligned}$$
\end{corollary}
\begin{proof}
This is simply because $\rank\Fcal_C+\rank\Ucal_C=\dim_\Qbb V=\dim_\Qbb V_0+d\dim_F\left(\bigoplus\limits_{m=1}^nL_m\right)$ and $V_0$ only contributes to the unitary part.
\end{proof}

\subsection{Proof of the main result in the unitary case}\label{proof of the main result in the unitary case}
Applying the formulas for the rank of Higgs bundles on a Shimura curve contained in a Shimura subvariety of $\SU(n,1)$-type given by a primary symplectic representation of type $\Lambda_m$,
we are able to prove our main theorem in the unitary case.
\begin{proof}[Proof of Theorem {\rm \ref{Main theorem}}]
From the inclusion of Shimura subdata
$$(\Gbf^W,X^W)\subset(\Gbf^H,X^H)\subset(\GSp_V,\Hscr_V)$$
with $V$ a symplectic $\Qbb$-space of dimension $2g$, $h:H\times H\ra E$ an Hermitian space subject to the natural constraints of signatures of $\SU(n,1)$-type and $H=U\oplus W$ an orthogonal decomposition with $U$ positive definite of dimension $n-1$ along the embedding $\tau:F\ra \Rbb$,
we get a Shimura curve $C$ in $M$ defined by $(\Gbf^W,X^W)$ of $\SU(1,1)$-type, with $(\Gbf^W)^\der=\Res_{F/\Qbb}\SU(W)$. The symplectic representation of $(\Gbf^W)^\der$ is restricted from $(\Gbf^H)^\der\mono\Sp_V$,
which is the scalar restriction from $F$ to $\Qbb$ of $\lambda:\SU(H)\mono\Sp_L$ for some symplectic representation over $L$,
such that $L\otimes_{F,\tau}\Rbb\isom\Lambda_m^{\oplus N}$ for some integer $N>0$.

In this case we have $\dim_\Qbb V=2g=2Nd\binom{n+1}{m}$ with $d=[F:\Qbb]$.
From Proposition \ref{decomposition over a curve} it follows that in the decomposition of the Higgs bundles over
the compactification $\Cbar$ of the Shimura curve $C$ defined by the Shimura datum $(\Gbf^W,X^W)$,
we have
\begin{equation}\label{eqnpfof main theorem 1}
\rank\Fcal_\Cbar^{\minuszero}=\dfrac{\dim_\Qbb V\cdot \binom{n-1}{m-1}}{d\cdot \binom{n+1}{m}}=\dfrac{2gm(n-m+1)}{dn(n+1)}.
\end{equation}
According to Theorem \ref{exclusion of Shimura curves}, if $\rank\Fcal_\Cbar^{-1,0}\leq(g-2)/5$, i.e.,
\begin{equation}\label{eqnpfof main theorem 2}
\left(1- \frac{10m(n-m+1)}{dn(n+1)}\right)\cdot g\geq 2,
\end{equation}
then $C$ is not contained generically in the Torelli locus $\Tcal_g$.

Now let $C'$ be a Hecke translate of $C$, given by $(q\Gbf^Wq^\inv,qX^W;qX^{W,+})$ for some $q\in(\Gbf^{H})^\der(\Qbb)^+$, where we take the derived group because the center of $\Gbf^H(\Qbb)^+$ does not contribute to the Hecke correspondences. Then by scalar restriction we have $q\in\SU(H)(F)$, which sends the orthogonal decomposition $H=U\oplus W$ into $H=qU\oplus qW$, with $qU$ positive definite along $\tau$, and $q^W$ of signature $(1,1)$ along $\tau$. Hence the arguments above gives the same inequality for $C'$. When $q$ runs through $(\Gbf^H)^\der(\Qbb)^+$, we obtain a Zariski dense subset of curves, none of which is contained generically in $\Tcal_g$, hence the conclusion follows from Lemma \ref{reduction to the open Torelli locus}.
\end{proof}

Before entering the proof to Corollary \ref{Main corollary}, we first deduce the following corollary from Theorem \ref{Main theorem}.
\begin{corollary}\label{pfof Main corollary}
Under the assumptions and notations of Theorem \ref{Main theorem}, the Shimura subvariety of $\SU(n,1)$-type is not contained generically in the Torelli locus $\calt_g$ if:
\begin{itemize}
\item $d\geq \frac{12}{n+1}$ when $m=1$;
\item $d\geq \frac{4(5n-4)}{n(n+1)}$ when $m=2$;
\item $d\geq \frac{12+30(n-1)(n-2)}{(n+1)n(n-1)}$ when $m=3$;
\item $d\geq3$ when $m\geq 4$.
\end{itemize}
\end{corollary}
\begin{proof}
By Theorem \ref{Main theorem}, it suffices to show that the condition \eqref{eqnassumption in main theorem} is satisfied under our assumptions.
Since $g=Nd\binom{n+1}{m}$, we have $g\geq d\binom{n+1}{m}$.
Hence \eqref{eqnassumption in main theorem} is satisfied if
\begin{equation}\label{eqnpfofmaincorollay 1}
d\geq \frac{2}{\binom{n+1}{m}}+\frac{10m(n-m+1)}{n(n+1)}.
\end{equation}
This is satisfied if $d\geq \frac{12}{n+1}$ when $m=1$, or $d\geq \frac{4(5n-4)}{n(n+1)}$ when $m=2$,
or $d\geq \frac{12+30(n-1)(n-2)}{(n+1)n(n-1)}$ when $m=3$.
To complete the proof,
it suffices to prove that if
\begin{equation}\label{eqnpfofmaincorollay 2}
\frac{2}{\binom{n+1}{m}}+\frac{10m(n-m+1)}{n(n+1)}\leq 3, \qquad \text{if~}m\geq 4.
\end{equation}
Note that $n\geq m\geq 4$. It is not difficult to check that \eqref{eqnpfofmaincorollay 2} is satisfied for $n\leq 8$.
Now consider the case $n\geq 9$. Since $\binom{n+1}{m}\geq n+1$ and $m(n+1-m)\leq \frac{(n+1)^2}{4}$,
one gets
$$\frac{2}{\binom{n+1}{m}}+\frac{10m(n-m+1)}{n(n+1)} \leq \frac{2}{n+1}+\frac{5(n+1)}{2n} \leq 3, \qquad \text{if~}n\geq 9$$which completes the proof.\end{proof}

\begin{proof}[Proof of Corollary {\rm\ref{Main corollary}}]
The case $m=1$ is already done. When $m\geq 2$, we have $n\geq m$, hence the conditions of Corollary \ref{pfof Main corollary} are satisfied if $d\geq 4$, or $d\geq 3$ and $n\geq 6$.
\end{proof}

\section{Symplectic representations in the orthogonal case}\label{section Symplectic representations in the orthogonal case}
In this section we treat the orthogonal case, i.e., symplectic representations of spin groups.

\subsection{Symplectic spinor representations}

We first recall some facts about split quadratic spaces and spinor representations, details of which are parallel to the case over $\Cbb$ in \cite[Chapter 6]{goodman wallach}.
\begin{definition}[(Split quadratic spaces)]\label{split quadratic spaces} A quadratic space $(H,Q)$ of dimensional $N$ over $F$ is said to be split over $F$ if:
\begin{list}{}
{\setlength{\labelwidth}{1.5cm}
\setlength{\leftmargin}{1.5cm}}
\item[$N$ even:] it is isomorphic to $T\oplus T^\vee$ for some $F$-vector space $T$ of dimension $N/2$ with $T^\vee$ its dual, such that $Q$ is equivalent to the pairing $((u,u^{\vee}),(v,v^\vee))\mapsto u^\vee(v)+v^\vee(u)$;

\item[$N$ odd:] it is isomorphic to $F\oplus T\oplus T^\vee$ for some $F$-vector space $T$ of dimension $(N-1)/2$ with $T^\vee$ its dual, such that $Q$ is equivalent to the pairing $((a,u,u^\vee),(b,v,v^\vee))\mapsto c_0ab+u^\vee(v)+v^\vee(u)$ for some constant $c_0\neq 0$. In this case $F$ is just a one-dimensional subspace orthogonal to $T\oplus T^\vee$.
\end{list}
The decompositions above are referred to as isotropic splittings of $(H,Q)$.
When $F$ is algebraically closed, any quadratic space is split.
It is also clear that when $(H,Q)$ is split over $F$, the $F$-group $\Spin(H)$ is split,
i.e., it admits a maximal $F$-torus which is split over $F$.
\end{definition}

When $(H,Q)$ is split over $F$ of dimension $N$, we have the following absolutely irreducible representations of $\Spin(H)$, called the spinor representations.
Using an $F$-basis $e_1,\cdots,e_m$ of $T$ and its dual basis $e_1^\vee,\cdots,e_m^\vee$, the $m$-dimensional split $F$-torus $\Tbf=\{(t_1,\cdots,t_m):t_j\in\Gbb_{\mrm, F} \}$
acts on $T$ by coordinates and on $T^\vee$ by the inverse coordinate functions: $t_j(a_je_j)=(t_ja_j)e_j$ and $t_j(b_je_j^\vee)=(t_j^\inv b_j)e_j^\vee$.
In this case $\Tbf$ is a split maximal $F$-torus in $\Spin(H)$, and in the corresponding root system  we can write down an explicit basis of simple roots $\epsilon_1,\cdots,\epsilon_m$
constructed out of the dual bases of $T$ and $T^\vee$, which describes the spinor representations  as follows:
\begin{list}{}
{\setlength{\labelwidth}{2.3cm}
\setlength{\leftmargin}{2.4cm}}
\item[$N=2m$:] the representation spaces are $P_+(T)=\wedge_F^{even}T$ and $P_-(T)=\wedge_F^{odd} T$, both of dimension $2^{m-1}$, and they are not isomorphic to each other; the highest weight of $P_\pm$ is $\omega_\pm=(\epsilon_1+\cdots+\epsilon_{m-1}\pm\epsilon_m)/2$ (they are also called the half spin representations); 

\item[$N=1+2m$:] the representation space is the full exterior $F$-algebra $P(T)=\wedge_FT$, of dimension $2^{m}$, with highest weight $\omega=(\epsilon_1+\cdots\epsilon_m)/2$. In particular, when $m=1$, $\Spin(H)\isom\SL_2$, and the spinor representation is isomorphic to the standard identity representation.  
\end{list}

Note that if $H=U\oplus T\oplus T^\vee$ is odd-dimensional with $U$ of dimension 1 orthogonal to $T\oplus T^\vee$,
then by restricting $Q$ to $H'=T\oplus T^\vee$ we get a quadratic subspace $H'$ in $H$ of codimension 1,
and in this case the restriction of the representation on $P(T)=\wedge_FT$ to $\Spin(H')$ becomes the direct sum of
the two spinor representations $P_+(T)$ and $P_-(T)$. Similarly, if $H=T\oplus T^\vee$ is even-dimensional,
with $T=U'\oplus T'$ and $T^\vee=U'^\vee\oplus T'^\vee$,
we may take a non-isotropic orthogonal decomposition $U'\oplus U'^\vee=U''\oplus U$ into two one-dimensional subspaces,
and form the $F$-subgroup $\Spin(H')$ for $H'=U\oplus T'\oplus T^\vee$.
In this case the restrictions of two spinor representations $P_+(T)$ and $P_-(T)$  to $\Spin(H')$
are both isomorphic to $P(T')$.
We refer to \cite[Chapter 6]{goodman wallach} for formulas of  weights of spinor representations using
which the decompositions above are verified directly.

In this setting the classification of Satake (cf. \cite[\S\,3.5]{satake embedding} and \cite[\S\,8.3]{satake classification}) gives:

\begin{theorem}[(Symplectic representations of spin groups)]\label{symplectic representations of spin groups}
{\rm (1)} Let $(H,Q)$ be a quadratic space of signature $(n,2)$ over $\Rbb$, and let $P$ resp. $P_\pm$ be the spin representations of $\Spin(H)_\Cbb$,
with $\dim_\Rbb H$ being odd resp. even. Then there are canonically defined Hermitian forms on these spaces of signature $(2^r,2^r)$ with $r=[(n+1)/2]$,
and up to isomorphisms the restriction $\Res_{\Cbb/\Rbb}P$ resp. $\Res_{\Cbb/\Rbb}P_\pm$ are the only irreducible symplectic representations of the $\Rbb$-group $\Spin(H)$, using the imaginary part of the Hermitian forms.

{\rm (2)} Let $(\Gbf,X)\subset(\GSp_V,\Hscr_V^\pm)$ be a Shimura subdatum of $\SO(n,2)$-type with $\Gbf^\der=\Res_{F/\Qbb}\Spin(H)$ for some quadratic space $(H,Q)$ over some totally real field subject to the constraints of signatures, such that the representation $\rho:\Gbf^\der\ra\Spin_V$ admits no trivial subrepresentations. Then $\rho$ is restricted from $\lambda:\Spin(H)\ra\Sp_L$ for some symplectic space $L$ over $F$, and that when we take the base change along $\tau$, $\lambda_\tau$ splits into a direct sum of copies of spin representations.

\end{theorem}

\begin{proposition}[(Unitary part over a Shimura subvariety of $\SO(n,2)$-type)]\label{unitary part over a Shimura subvariety of SO(n,2)-type}
Let $M\subset\Acal_V$ be a Shimura subvariety of $\SO(n,2)$-type, i.e., defined by some Shimura subdatum $(\Gbf,X;X^+)\subset(\GSp_V,\Hscr_V^\pm;\Hscr_V^+)$ of $\SO(n,2)$-type as in Theorem \ref{symplectic representations of spin groups}{\,\rm(2)}, using a quadratic space $(H,Q)$ over a totally real number field $F$.
Assume that the symplectic representation $\rho:\Gbf^\der\isom\Res_{F/\Qbb}\Spin(H)\mono\Sp_V$ admits no trivial subrepresentations,
and that $\Ecal_M=\Fcal_M\oplus\Ucal_M$ is the decomposition of the Higgs bundles on $\Ecal_M$ induced by the universal family over $\Acal_V$, with $\Ucal_M$ the maximal unitary Higgs subbundle.
Then
$$\rank\Ucal_M^\minuszero=\frac{(d-1)g}{d},\quad\rank\Fcal_M^\minuszero=\frac{g}{d},\qquad \text{with~}d=[F:\Qbb].$$
\end{proposition}

\begin{proof}
The proof is exactly the same as Proposition \ref{primary decomposition over a Shimura subvariety of SU(n,1)-type}
because when we take the base change $\Qbb\ra\Rbb$ we get $\rho_\Rbb=\bigoplus\limits_{j=1,\cdots,d}\lambda_j$ with $\rho=\Res_{F/\Qbb}\lambda$ and $\lambda_j=\lambda\otimes_{F,\tau_j}\Rbb$,
using the same notations as in \ref{symplectic representations of spin groups}(2), and only the component $\lambda_1=\lambda_\tau$ contributes to $\Fcal_M$.
\end{proof}

\begin{corollary}[(Formula of $\SO(n,2)$-type)]\label{formula of SO(n,2)-type}
Let $M\subset\Acal_V$ be a Shimura subvariety of $\SO(n,2)$-type, given by a subdatum $(\Gbf,X;X^+)\subset(\GSp_V,\Hscr_V^\pm;\Hscr_V^+)$,
whose symplectic representation $\rho:\Gbf^\der\isom\Res_{F/\Qbb}\Spin(H)\ra\Sp_V$ is of the form $V=V_0\oplus\Res_{F/\Qbb}\lambda$
with $V_0$ a trivial subrepresentation and $\lambda$ some $F$-linear symplectic representation $\lambda:\Spin(H)\ra\Sp_L$ admitting no trivial subrepresentation.
Then in the decomposition of Higgs bundles $\Ecal_M=\Fcal_M\oplus\Ucal_M$, we have
$$\rank\Ucal_M^\minuszero=\frac{\dim_FV_0+(d-1)\dim_FL}{2},\quad \rank\Fcal_M^\minuszero=\frac{\dim_FL}{2},\qquad \text{with~}d=[F:\Qbb].$$
\end{corollary}

\begin{proof}
This is similar to Corollaries \ref{formula of SU(n,1)-type} and \ref{formula of SU(n,1)-type}, using the fact that the spinor representations are the only non-trivial symplectic representations over $\Rbb$: when $n$ is odd, there is only one spinor representation; when $n$ is even, the two spinor representations over $\Rbb$ are of the same dimensions and same signatures.
\end{proof}

\subsection{Restriction to a Shimura curve}
We first consider the case over $\Rbb$, i.e., study the decomposition of the irreducible symplectic representations of $\Spin(H)$ to $\Spin(W)$ over $\Rbb$, where $H$ is of signature $(n,2)$ and $W$ is of signature $(1,2)$. We assume that $W=U'\oplus W'$ with $U'$ positive definite of dimension 1 and $W'$ negative definite of dimension 2 together with an isotropic splitting over $\Cbb$ given by $W'_\Cbb\isom T_1\oplus T_1^\vee$, and we take a further orthogonal decomposition $H=W\oplus U$ for some definite subspace $U$ of dimension $n-1$. We also assume that
\begin{list}{}
{\setlength{\labelwidth}{1.3cm}
\setlength{\leftmargin}{1.4cm}}
\item[$n$ odd:] $U_\Cbb$ admits an isotropic splitting $U_\Cbb=T_2\oplus T_2^\vee$;
\item[$n$ even:] $U=U''\oplus W''$ with $U''$ one-dimensional orthogonal to $W''$, and $W''_\Cbb=T_2\oplus T_2^\vee$ is an isotropic splitting;
\end{list}

If $n\geq 1$ is odd, then we can take an isotropic splitting $U_\Cbb=T_2\oplus T_2^\vee$,
and it turns out that $T=T_1\oplus T_2$ gives a maximal isotropic decomposition $H_\Cbb=U'_\Cbb\oplus T\oplus T^\vee$.
Thus we have $P(T)\isom P(T_1)\otimes P(T_2)$ on which the Levi $\Rbb$-subgroup $\Spin(W)\times\Spin(U)$ acts,
and the restriction to $\Spin(W)$ is a direct sum of copies of $P(T_1)$.

If $n\geq 2$ is even, then $H=U''\oplus H'$ is an orthogonal decomposition with  $H'=U'\oplus W'\oplus W''$ odd-dimensional. The two spinor representations $P_\pm$ of $\Spin(H)$ both restrict to the spin representation $P(T_1\oplus T_2)$ of $\Spin(H')$, and restricts further to a direct sum of copies of $P(T_1)$ for $\Spin(W)$ because we are reduced to the odd-dimensional case.

We have thus seen that in both cases no additional trivial subrepresentation occurs when restricted to $\Spin(W)$, hence we arrive at the main theorems in the orthogonal case:

\begin{proposition}[(Decomposition over a curve in the orthogonal case)]\label{decomposition over a curve in the orthogonal case}
Let $M\subset\Acal_V$ be a Shimura subvariety of $\SO(n,2)$-type, given by a subdatum $(\Gbf,X)\subset(\GSp_V,\Hscr_V^\pm)$.

{\rm (1)} Assume that the symplectic representation $\rho:\Gbf^\der=\Res_{F/\Qbb}\Spin(H)\ra\GSp_V$ admits no trivial subrepresentations. Let $C\subset M$ be a Shimura curve cut out by an $F$-subspace $U$ of $H$ positive definite of dimension $n-1$ along $\tau$ following the notations of Definition \ref{Shimura data of SO(n,2)-type}, and let $\Ecal_C$ be the Higgs bundle on $C$ induced by the universal abelian schemes on $\Acal_V$, with $\Ecal_C=\Fcal_C\oplus\Ucal_C$ the decomposition into Higgs subbundles, $\Ucal_C$ being the maximal unitary part. Then
$$\rank\Ucal_C^\minuszero=\frac{(d-1)g}{d},\quad\rank\Fcal_C^\minuszero=\frac{g}{d},\qquad \text{with~}d=[F:\Qbb].$$

{\rm (2)} Similarly, when $\rho:\Gbf^\der\ra\Sp_V$ decomposes as $V=V_0\oplus\Res_{F/\Qbb}\lambda$ with $V_0$ a trivial subrepresentation and $\lambda:\Spin(H)\ra\Sp_L$ an $F$-linear symplectic representations admitting no trivial subrepresentation, then
$$\rank\Ucal_C^\minuszero=\frac{\dim_FV_0+(d-1)\dim_FL}{2},\quad \rank\Fcal_C^\minuszero=\frac{\dim_FL}{2},\qquad \text{with~}d=[F:\Qbb].$$
\end{proposition}

\begin{proof}
(1) The formula for $\rank\Fcal_C^\minuszero$ is immediate, as in this case no additional unitary part occurs when the representation is restricted from $\Spin(H_\tau)$ to $\Spin(W_\tau)$.
(2) is immediate from (1) and Corollary \ref{formula of SO(n,2)-type}.
\end{proof}

\begin{proof}[Proof of Theorem {\rm \ref{main theorem in the orthogonal case}}]
Following the notations in Proposition \ref{decomposition over a curve in the orthogonal case}\,(2), the inequality $\rank \Fcal^\minuszero_C=(2g-\dim_\Qbb V_0)/2d\leq(g-2)/5$ holds when $d\geq 6$.
In fact this is automatic if $d\geq 7$, and when $d=6$, it still holds because $g=\dim_\Qbb V/2$ is at least $d\cdot2^{[n+2]/2}$.\end{proof}

\begin{remark}
If in contrary that the totally real field $F$ is $\Qbb$, i.e., $d=1$, our arguments do not work for Shimura subvarieties of $\SO(n,2)$-type.
In this case the results in \cite{luzuo14} do provide some complments:
the symplectic representation does not give rise to unitary Higgs bundles when restricted to Shimura curves provided that there is no trivial subrepresentation, i.e., $V_0=0$
\big(cf. Proposition \ref{decomposition over a curve in the orthogonal case}\,(2)\big),
hence such curves, as well as such Shimura subvarieties, are not contained generically in $\Tcal_g$ for $g>11$.
\end{remark}

\end{document}